\documentclass[12pt]{amsart}

\usepackage{amsmath, enumerate, amsthm, amssymb, textcomp, booktabs}
\usepackage{anysize}

\usepackage[shortlabels]{enumitem}
\usepackage{url, hyperref}

\marginsize{2cm}{2cm}{1.5cm}{1cm}

\hypersetup{citecolor=blue, linkcolor=blue, colorlinks=true}

\setlist[description]{labelindent=0em, leftmargin=!}

 \newtheorem{theorem}{Theorem}[section]
 \newtheorem{lemma}[theorem]{Lemma}
 \newtheorem{corollary}[theorem]{Corollary}

\theoremstyle{remark}
\newtheorem{definition}[theorem]{Definition}
\newtheorem{remark}[theorem]{Remark}

\newcommand\PG{\mathsf{PG}}

\newcommand\PGL{\mathsf{PGL}}

\newcommand\Aut{\mathsf{Aut}}

\newcommand{\cyclic}{\mathcal{C}}

\newcommand{\axial}{order-preserving axial}

\renewcommand\ge{\geqslant}

\title{Simpler foundations for the hyperbolic plane}

\author{John Bamberg}
\address{ %
Department of Mathematics and Statistics\\
The University of Western Australia\\
Crawley, WA 6009, Australia.}
\email{john.bamberg@uwa.edu.au}

\author{Tim Penttila}
\address{ %
School of Mathematical Sciences\\
The University of Adelaide\\
Adelaide, SA 5005, Australia.}
\email{tim.penttila@adelaide.edu.au}

\begin{document}

\begin{abstract}
H. L. Skala (1992) gave the first elegant first-order axiom system for hyperbolic geometry by 
replacing Menger's axiom involving projectivities with the theorems of Pappus and Desargues for the hyperbolic plane.
In so doing, Skala showed that hyperbolic geometry is incidence geometry.
We improve upon Skala's formulation by doing away with Pappus and Desargues altogether, by substituting for them two simpler
axioms. 
\end{abstract}


\maketitle

\section{Introduction}

The independent discovery by Lobachevsky and Bolyai of hyperbolic geometry in the 1830's was
followed by slow acceptance of the subject from the 1860's on, with the publications of relevant parts of the
correspondence of Gauss. A new phase was entered from 1903, when Hilbert \cite{Hilbert:1903gf}, in his work introducing
the \emph{calculus of ends}, introduced an axiomatisation for hyperbolic plane geometry by adding a hyperbolic parallel
axiom to the axioms for plane absolute geometry. (A detailed survey of axiomatisations for hyperbolic geometry
is given in \cite{Pambuccian:2006mz}).

In 1938, Menger \cite{Menger:1938ul} made the important discovery that in hyperbolic geometry the concepts of 
betweenness and equidistance can be defined in terms of point-line incidence. Since an axiom system obtained by replacing all occurrences of betweenness and equidistance with their definitions in terms of incidence would look highly unnatural, Menger and his students looked for a more natural axiom system. 

\begin{quote}
\emph{``This research, which is still in progress, will result, I believe, in a system of fairly simple assumptions in terms of the projective operations for the whole Bolyai-Lobachevsky geometry.''}\hfill{\footnotesize{Karl Menger, 1940 \cite{Menger:1940ly}}.}
\end{quote}

This task was carried out by Menger \cite{Menger:1938ul,Menger:1940ly,Menger:1946gf,Menger:1971ve}, and by his students Abbott  \cite{Abbott:1941qv,Abbott:1942ty,Abbott:1944nr}, DeBaggis \cite{De-Baggis:1946vn,DeBaggis:1948rt}  and Jenks \cite{Jenks:1940bh,Jenks:1940pd,Jenks:1941lq}, but even their most polished axiom system contained a statement on projectivities that was not reducible to a first-order statement. In 1992, Menger's student Skala \cite{Skala:1992rr} showed that the
statement on projectivities can be replaced by the axioms of Pappus and Desargues for the hyperbolic plane, thereby producing the first elegant first-order axiom system for hyperbolic geometry based on incidence alone. Since one of the Menger school's other axioms was `Pascal's Theorem' on the rimpoints, this axiom system looks a little excessive. In 1905, Hessenberg \cite{Hessenberg:1905mi} had shown that, in a projective plane, the axiom of Pappus implies
the axiom of Desargues. In 1966, Buekenhout \cite{Buekenhout:1966hc} had shown that, in a projective plane, the axiom of Pascal for an oval
implies the axiom of Pappus for the plane. This context suggests that the axiom of Pascal for the rimpoints should
be sufficient without the need for the axioms of Pappus and Desargues. This paper establishes that suggestion to be valid, fulfilling Menger's dream, but for
two additional axioms that ensures that midpoints exist, and that perpendicularity of lines is well-defined.

We will be considering an incidence structure that fulfils the axioms (A1) through (A7) plus (A10) and (A11) below. The first
four axioms allow us to have a ternary relation of \emph{betweenness}.
For three collinear points $A, B, C$,  the point $B$ is said to be {\bf between} $A$ and $C$, if every line through $B$ intersects at least one line of each pair of intersecting lines which pass through $A$ and $C$, respectively. Using the concept of betweenness one can define \emph{ray} and \emph{segment} in the usual manner. Two rays $r$ and $s$ with endpoints $R$ and $S$, respectively, are said to be (strictly) {\bf parallel} if $r$ and $s$ are noncollinear and if every line that meets one of the rays meets either the other ray or the segment $\overline{RS}$. Two lines, or a ray and a line, are said to be parallel if they contain rays that are parallel. 
Then it can be shown that there are exactly two lines through a given point which are parallel to a given line \cite[Corollary 2.4]{DeBaggis:1948rt}. By (A6) below, 
two non-collinear rays have a common parallel line.

A pair $(a,b)$ of parallel lines is called a \textbf{rimpoint} if there is a line $\ell$
parallel to both $a$ and $b$, and there is a line which intersects $a$, $b$ and $\ell$.
A rimpoint $(a,b)$ is equal to another rimpoint $(c,d)$ if every pair of lines from
$\{a,b,c,d\}$ are identical or parallel. A rimpoint $(a,b)$ is incident with a line $\ell$ if $\ell=a$, $\ell=b$ or $\ell$ is parallel to both $a$ and $b$,
and there exists a line that intersects $a$, $b$, and $\ell$. From (A1) -- (A7) below, every line of our incidence structure
has two rimpoints \cite[Theorem 1.7]{Abbott:1941qv}. It is useful to the reader to have in mind the examples of hyperbolic
planes characterised by our main result: the Cayley-Klein hyperbolic planes over Euclidean fields. Thus, one constructs
such a geometry by first considering the unit circle $\mathcal{O}$ centred at the origin of the Cartesian plane $\mathbb{E}^2$, where $\mathbb{E}$ is a Euclidean field.
The finite points of the hyperbolic plane are the points of $\mathbb{E}$ lying within the interior of circle,
the lines of the hyperbolic plane are the chords of the circle, and the rimpoints are (essentially) the points of the circle. Two lines are parallel if they intersect
in a point on the circle, and two lines are perpendicular if they are conjugate with respect to the circle.

We will use the term `point' for both a point 
and a rimpoint, and we will refer to points of our incidence structure as `finite points'. 
If $A$ and $B$ are points, then we write $AB$ for the unique line incident with both $A$ and $B$.
We will often write a polygon in an order that distinguishes sides of the polygon. For example, 
$ABCDEF$ is a hexagon with distinguished sides $AB$, $BC$, $CD$, and so forth.

\begin{theorem}[Skala \cite{Skala:1992rr}]\label{Skala}
Every model of axioms (A1) through (A9) is a Cayley-Klein hyperbolic plane over a Euclidean field.

\begin{enumerate}[{\rm ({A}1)}]
\item Any two distinct points are incident with a unique line.
\item Each line is on at least one point.
\item There exist three non-collinear points, and three collinear points.
\item (Betweenness) Of three collinear points, at least one has the property that every line through 
it intersects at least one of each pair of intersecting lines through the other two. 
\item (Hyperbolicity) If $P$ is not on $\ell$, then there exist two distinct lines on $P$ not meeting $\ell$ and such that each line meeting $\ell$ meets at least one of those two lines.
\item (Parallelism) Any two non-collinear rays have a common parallel line.
\item (Pascal's Theorem on rimpoints)
If $A$, $B$, $C$, $D$, $E$, $F$ are rimpoints such that the three diagonal points $AB\cap DE$, $BC\cap EF$, $CD\cap FA$
of the hexagon $ABCDEF$ exist, then the diagonal points are collinear.
\item (Pappus' Theorem) Let $a$ and $b$ be lines containing points $A_1$, $A_2$, $A_3$, and $B_1$, $B_2$, $B_3$, respectively. If the points 
$A_1B_2\cap A_2B_1$, $A_1B_3\cap A_3B_1$, $A_2B_3\cap A_3B_2$ exist, then they are collinear.
\item (Desargues' Theorem)  Let the lines $a$, $b$, $c$ be concurrent in a point and contain pairs of points 
$\{A_1, A_2\}$, $\{B_1,B_2\}$ and $\{C_1, C_2\}$, respectively. If the points
$A_1B_1\cap A_2B_2$, $A_1C_1\cap A_2C_2$, $B_1C_1\cap B_2C_2$ exist, then they are collinear.
\end{enumerate}
\end{theorem}

As Pambuccian observed in footnote 7 of \cite{Pambuccian:2001ve}, Skala's original Axiom 4 follows from her Axiom 6. That is
why we have presented only nine of her ten axioms here. We should also point out that Skala's statement of (A9) was incorrect, 
but that her correct statement of the `Dual of Desargues' Theorem' \cite[p. 259]{Skala:1992rr} remedied this anomaly. We will be considering the following axiom in our formulation:

\begin{enumerate}
\item[\rm (A10)] \textit{
Given distinct finite points $P$, $Q$ there exist lines $\ell$ on $P$ and $m$ on $Q$ such that
two common parallels to $\ell$ and $m$ meet on $PQ$.}
\end{enumerate}

This axiom allows us to define (unique) midpoints of line segments in our incidence structure: the \emph{midpoint} of the segment
$\overline{PQ}$ is defined to be the point of intersection $\ell\cap m$ in the statement of (A10). Moreover (A10) leads us to
a new notion of perpendicularity of lines (see Definition \ref{defperp}), that is independent of a notion of congruence, and 
based purely on the properties of \emph{axial maps}.
This approach was motivated by the work of Rigby \cite{Rigby:1969ab}. Our relation on lines has the usual properties
of perpendicularity that holds for Cayley-Klein hyperbolic planes: (a) two non-intersecting non-parallel lines have a unique common perpendicular line,
(b) for each point $P$ and line $\ell$ there exists a unique line through $P$ perpendicular to $\ell$, (c) there are no rectilaterals.
To ensure that our definition (in a later section) of perpendicularity is well-defined, we need another axiom. In the following, if $X$ is a rimpoint
and $A$ is a finite point, then $X^A$ is the other rimpoint on the line $XA$ joining $X$ and $A$.

\begin{enumerate}
\item[\rm (A11)]\textit{
Let $X$ be a rimpoint not on a line $\ell$. 
Then there is a unique point $R$ on $\ell$ such that for every finite point $S$ on $\ell$,
 two common parallels to $\ell$ and $X^{SR}X^{RS}$ meet on $XR$.}
\end{enumerate}

We show that we can substitute $\mathrm{(A8)}\wedge \mathrm{(A9)}$ for the weaker 
$(\mathrm{A10})\wedge (\mathrm{A11})$, in order to 
completely characterise hyperbolic planes over Euclidean fields.

\begin{theorem}\label{main}
Every model of axioms (A1) through (A7), plus (A10) and (A11), is a hyperbolic plane and conversely. 
Moreover, these models are precisely the Cayley-Klein hyperbolic planes over Euclidean fields.
\end{theorem}

The proof of this result appears in Section \ref{section:UsingBachmann} and it draws together the ideas of Bachmann \cite{Bachmann:1973aa}, Buekenhout \cite{Buekenhout:1966hc} and Rigby \cite{Rigby:1969ab} 
to develop a theory of half-turns and reflections for an incidence structure satisfying the above axioms. Thus Sections \ref{section:cyclicorder}, \ref{section:halfturns}, \ref{section:definedperp},
\ref{section:gliderefs}, \ref{section:UsingBachmann} develop this extra structure on our incidence structure. 
We show that axioms (A1) -- (A7) plus (A10) and (A11) enable us to affirm the axioms of a \emph{metric plane}, 
and then with a bit more work, we use Bachmann's characterisation theorem for hyperbolic planes 
(Theorem \ref{bachmann}) to establish an isomorphism to a Cayley-Klein hyperbolic plane over a Euclidean field.
An alternative proof will be given in Section \ref{section:abstractovals} that uses the theory of abstract ovals.

% The Klingenberg planes are generalisations of hyperbolic
%planes over Euclidean fields (in fact, they are hyperbolic planes over arbitrary ordered fields),
%and they need not satisfy (A10). However, they also do not have \emph{rimpoints}, and so we leave it
%as an open problem whether or not our axioms are independent.

\subsection*{Some notation and definitions}

In this paper, we will often encounter standard concepts in incidence geometry and group theory.
An \textbf{incidence structure} is a triple $(\mathcal{P},\mathcal{B},\mathrm{I})$ with $\mathcal{P}$ and $\mathcal{B}$
disjoint, non-empty sets, and $\mathrm{I}$ a relation between them; the \emph{incidence relation}. An \textbf{isomorphism} from the incidence structure $(\mathcal{P},\mathcal{B},\mathrm{I})$ to 
the incidence structure $(\mathcal{P'},\mathcal{B'},\mathrm{I'})$ is a bijection $\phi\colon\mathcal{P}\cup\mathcal{B}\to\mathcal{P'}\cup\mathcal{B'}$ such that $\mathcal{P}^\phi=\mathcal{P'}$, $\mathcal{B}^\phi=\mathcal{B'}$, and for all $P\in\mathcal{P}$ and $B\in\mathcal{B}$:
\[
P\,\mathrm{I}\, B \Leftrightarrow P^\phi\, \mathrm{I'}\, B^\phi.
\]
An \textbf{automorphism} of an incidence structure $\mathcal{S}=(\mathcal{P},\mathcal{B},I)$ is an isomorphism to itself. We denote the set of all automorphisms of $\mathcal{S}$ by $\Aut(\mathcal{S})$. Moreover, $\Aut(\mathcal{S})$ is a group under composition, called the \textbf{automorphism group} of $\mathcal{S}$.
If $G$ is a group acting on a set $\Omega$, and $H$ is a group acting on a set $\Gamma$, we say that these two actions
are \textbf{permutationally isomorphic}, and we write $(\Omega,G)\cong (\Gamma, H)$, if there exists a group isomorphism $\varphi:G\to H$ and a bijection $\beta:\Omega\to \Gamma$ such
that
$\beta(\omega^g) = \beta(\omega)^{\varphi(g)}$
for all $\omega\in\Omega$ and $g\in G$. The pair $(\beta, \varphi)$ is a \textbf{permutational isomorphism} from $(\Omega,G)$ to $(\Gamma, H)$, and permutational isomorphism
is an equivalence relation on group actions.

%%%%%%%%%%%%%%%%%%%%%%%
%
%  	Cyclic order and separation
%
%%%%%%%%%%%%%%%%%%%%%%%

\section{Cyclic order and separation}\label{section:cyclicorder}

A nonempty set $\Omega$ is said to be \emph{cyclically ordered} if it is equipped with a ternary
relation $\cyclic$ that possesses the following properties (see \cite{Huntington:1935rz}):
\begin{enumerate}[({C}1)]
\item if $\cyclic(x,y,z)$, then $x\ne y\ne z\ne x$;
\item if $x\ne y\ne z\ne x$, then exactly one of the relations $\cyclic(x, y, z)$ or $\cyclic(x, z, y)$ holds; 
\item if $\cyclic(x, y, z)$, then $\cyclic(y, z, x)$;
\item if $\cyclic(x, y, z)$ and $\cyclic(x, z, t)$, then $\cyclic(x, y, t)$.
\end{enumerate}

The \emph{reverse} $\bar{\cyclic}$ of $\cyclic$ is the order $\bar{\cyclic}(x,y,z)$ if and only if
$\cyclic(z,y,x)$.  
%An \emph{automorphism} $\tau$ of a cyclically ordered set $(\Omega,\cyclic)$ is a
%permutation of $\Omega$ that respects the relation $\cyclic$. That is, if $\cyclic(x,y,z)$, then
%$\cyclic(x^\tau,y^\tau,z^\tau)$. The automorphisms of $(\Omega,\cyclic)$ form a group, and a
%subgroup $G$ of automorphisms of a cyclically ordered set $(\Omega,\cyclic)$ is said to be
%\emph{$c$-3-transitive} if for any elements $x_1,y_1,x_2,y_2,x_3,y_3$ satisfying
%$\cyclic(x_1,x_2,x_3)$ and $\cyclic(y_1,y_2,y_3)$, there exists an element $g\in G$ such that
%$x_i^g=y_i$ for all $i\in\{1,2,3\}$.

The relation of \emph{separation} is a (nonempty) quaternary relation on a set $\Omega$ if the
following properties hold on five elements $A$, $B$, $C$, $D$, $X$ of $\Omega$ (see
\cite{Huntington:1935rz} and \cite{Pambuccian:2011bh}):
\begin{enumerate}[(S1)]
\item if $A$, $B$ separates $C$, $D$, then $A$, $B$, $C$, $D$ are distinct;
\item if $A$, $B$ separates $C$, $D$, then $A$, $B$ separates $D$, $C$; % $R(A,B,C,D)$, then $R(B,C,D,A)$;
\item if $A$, $B$ separates $C$, $D$, then it is not the case that $A$, $C$ separates $B$, $D$; % $R(A,B,C,D)$, then $R(A,B,D,C)$ does not hold;
\item $A$, $B$ separates $C$, $D$ or $A$, $C$ separates $B$, $D$ or $A$, $D$ separates $B$, $C$;
%$R(A,B,C,D)$ or $R(A,C,B,D)$ or $R(A,B,D,C)$; %  $A$, $B$ separates $C$, $D$ or $A$, $C$ separates $B$, $D$ or $A$, $D$ separates $B$, $C$;
\item  If $A$, $B$ separates $C$, $D$, and $A$, $C$ separates $B$, $X$, then $A$, $B$ separates $D$, $X$.
%if $R(A,X,C,D)$ and $R(A,B,C,X)$, then $R(A,B,C,D)$.
\end{enumerate}

We will often use the notation $AB // CD$ for ``$A$, $B$ separates $C$, $D$'', and it will soon
become clear why this notation suits our purposes.  Given a cyclic order $\cyclic$ on $\Omega$, we
can construct a separation relation on $\Omega$ by the following rule:
\[
AC// BD\iff \left( \cyclic(A,B,C) \wedge \cyclic(C,D,A) \right)\vee \left( \cyclic(A,D,C)\wedge \cyclic(C,B,A) \right).
\]
Conversely, given a separation relation on a set $\Omega$, we can associate two cyclic orders on
$\Omega$, each of which is the reverse of the other.

\begin{definition}\label{defsep}
Let $\mathcal{I}$ be an incidence structure of points and lines satisfying (A1) through to (A7).
Let $\mathcal{O}$ be the set of rimpoints of $\mathcal{I}$.  For four rimpoints $A$, $B$, $C$, $D$,
we write $AB//CD$ if the line $AB$ meets $CD$ in a finite point.
\end{definition}
 
 We will show that the above quaternary relation on rimpoints satisfies the axioms for a separation
 relation. It is clear that Axioms (S1) and (S2) hold, and Theorem \ref{seprelation} shows that
 (S3), (S4) and (S5) hold.  The \emph{Pasch property} follows from the first five axioms (A1)--(A5)
 \cite[p. 168]{Pambuccian:2008pd}, and is the culmination of the work of Menger's students, from
 DeBaggis \cite{De-Baggis:1946vn} and Jenks \cite{Jenks:1940bh}, to Abbott \cite{Abbott:1941qv} (for
 the \emph{closed} plane; the finite points and the rimpoints).

\begin{theorem}[Pasch property ({\cite[Theorem 1.11]{Abbott:1941qv}})]\label{thm:Pasch}
If $PQR$ is any triangle in $\mathcal{I}$ and if $\ell$ is any line which intersects $PQ$, then
$\ell$ contains $R$ or intersects exactly one other side in an interior point.
\end{theorem}

\begin{corollary}\label{cor:Pasch}
Let $A, B, C, M, N$ be distinct rimpoints. Then $\{M, N\}$ separates two or none of the pairs $\{A,
B\}$, $\{B, C\}$, $\{C, A\}$.
\end{corollary}

\begin{proof}
Suppose $\{M,N\}$ separates at least one of the pairs. Without loss of generality, we may assume
that $\{M,N\}$ separates $\{A,B\}$. Therefore, $MN$ and $AB$ meet in a finite point $P$.  Consider
the triangle $ABC$. Then the line $MN$ intersects a side of $ABC$ and so by Theorem \ref{thm:Pasch},
it intersects exactly one of $BC$ or $CA$ (n.b., $MN$ does not contain a vertex as $M,N\ne
A,B,C$). So it follows immediately that $\{M, N\}$ separates one of $\{B, C\}$ or $\{C, A\}$.
\end{proof}

Following Abbott \cite{Abbott:1941qv}, given a line $\ell$, we say that two points $A$ and $B$ that
are not incident with $\ell$ are \emph{on the same side of $\ell$} if one of the following holds:
\begin{enumerate}[(i)]
\item $A=B$,
\item $A$ and $B$ are rimpoints and $AB$ does not intersect $\ell$ in a finite point,
\item $A$ and $B$ are finite points and the segment $\overline{AB}$ does not have a finite point in
  common with $\ell$,
\item $A$ is finite and $B$ is a rimpoint, and if $X$ and $X'$ are the two rimpoints not incident
  with $\ell$ such that $AX$ and $AX'$ are the two lines through $A$ parallel to $\ell$, then $B$
  lies on the same side of $\ell$ as $X,X'$.
\item $A$ is a rimpoint and $B$ is finite, and we apply (iv) to $B$ and $A$.
\end{enumerate}
The property (iv) is well-defined since the two rimpoints $X$ and $X'$ on the lines through $A$
parallel to $\ell$ satisfy the property (ii) \cite[Theorem 1.2]{Abbott:1941qv}. We write $A\sim_\ell
B$ for this relation `on the same side of $\ell$'.

\begin{lemma}[{\cite[Theorem 1.1]{Abbott:1941qv}}]\label{eqrelation}
The relation $\sim_\ell$ is an equivalence relation with two classes.
\end{lemma}

The two equivalence classes of $\sim_\ell$ are the two \emph{sides} of $\ell$.

\begin{remark}\label{remark:differentsides}
Another way to express this equivalence relation $\sim_\ell$ just for the rimpoints is by using a
cyclic order $\cyclic$ associated to the separation relation $//$. It is not difficult to see that
if the rimpoints of $\ell$ are $L$ and $L'$, then rimpoints $A$ and $B$ lie on the same side of
$\ell$ if and only if $\cyclic(L,A,L')$ and $\cyclic(L,B,L')$, or $\cyclic(L',A,L)$ and
$\cyclic(L',B,L)$.
\end{remark}

\section{Half-turns and axial permutations}\label{section:halfturns}

Throughout this section, we will suppose that $\mathcal{I}$ is an incidence structure satisfying
(A1) through (A7), and we will also be adopting (A10) later in this section, but not immediately.

\subsection{Basic properties of half-turns}

Given a point $P$, the \emph{half-turn} about $P$ is a permutation of $\mathcal{O}$ defined in the
following way: If $X\in\mathcal{O}$, then $X^{P}$ is the unique element of $\mathcal{O}$ incident
with $XP$, not equal to $X$. Therefore, if $A$ and $B$ are rimpoints and $P$ is a finite point on
$AB$, then $A^{P}=B$. We will use the same letter for a finite point and its half-turn, as is
standard in transformation geometry.  Each half-turn is determined by two of its transpositions in
its decomposition into disjoint transpositions, for if $(A\, B)$ and $(C\, D)$ are two such
components of the half-turn about $X$, then the point $X$ is the intersection of $AB$ and $CD$. We
use the symbol `$1$' to denote the identity permutation on the rimpoints.  Theorem
\ref{inv_then_collinear} is the analogue of \cite[Proposition 7.2]{Buekenhout:1966hc} for abstract
ovals.

\begin{theorem}\label{inv_then_collinear}
If $P$, $Q$, $R$ are finite distinct points such that $PQR$ is an involution, then $P$, $Q$, $R$ are
collinear.
\end{theorem}

\begin{proof}
Let $A\in \mathcal{O}$. Let $B:=A^{P}$, $C:=B^{Q}$, $D:=C^{R}$, $E:=D^{P}$, $F:=E^{Q}$. Then
$F^{R}=A$ since $PQRPQR=1$. Thus the points $P=AB \cap DE$, $Q=BC \cap EF$ and $R=CD \cap FA$
exist, and are collinear by (A7).
\end{proof}

\begin{lemma}\label{interchangesides}
Any half-turn about a point $P$ on a line $\ell$ interchanges the sides of $\ell$.
\end{lemma}

\begin{proof}
Let $Q$ be a rimpoint not incident with $\ell$. By definition, $Q^P$ is the unique element of
$\mathcal{O}$ on $PQ$ not equal to $Q$. So $QQ^P$ and $\ell$ intersect in a finite point (namely
$P$), and therefore, $Q^P$ lies on the opposite side of $\ell$ as $Q$.
\end{proof}

%\begin{lemma}\label{cyclic_one_side}
%Let $\cyclic$ be one of the two cyclic orders associated to the separation relation $//$.  Suppose
%$X$, $X'$, $Z$ are rimpoints and $P$ is a finite point on $XX'$. If $\cyclic(X, X', Z)$, then
%$\cyclic(X, Z^P, X')$.
%\end{lemma}
%
%\begin{proof}
%Suppose $\cyclic(X,X',Z)$. Since $Z$ and $Z^P$ lie on different sides of $\ell$, we have by Remark
%\ref{remark:differentsides}, $\neg \cyclic(X,X',Z^P)$ and hence $\cyclic(X,Z^P,X')$ (by (C2)).
%\end{proof}

\begin{lemma}\label{commutinginvolutions}\samepage
Let $P$ and $Q$ be finite points. 
\begin{enumerate}[(a)]
\item If $P\ne Q$, then for all rimpoints $X\in\mathcal{O}$ such that $X$ does not lie on the line
  $PQ$, we have that $X$ and $X^{PQPQ}$ lie on different sides of the line joining $X^{P}$ and
  $X^{PQ}$.
\item If the permutations $PQ$ and $QP$ are equal, then $P=Q$.
\end{enumerate}
\end{lemma}

\begin{proof}
For the proof of (a), let $X$ be a rimpoint not incident with $PQ$. Let $\ell$ be the line spanned
by $X^{P}$ and $X^{PQ}$. Now the line joining $X^{P}$ and $X$ is incident with $P$ and so $X$ and
$P$ lie on the same side of $\ell$. Similarly, the line joining $X^{PQP}$ and $X^{PQ}$ is incident with $P$ and so
$X^{PQP}$ and $P$ lie on the same side of $\ell$. Therefore, by Lemma \ref{eqrelation}, $X$ and
$X^{PQP}$ lie on the same side of $\ell$. Now the line spanning $X^{PQP}$ and $X^{PQPQ}$ passes
through $Q$, and $Q$ is incident with $\ell$, so $X^{PQP}$ and $X^{PQPQ}$ lie on different sides of
$\ell$ (by Lemma \ref{interchangesides}). It then follows from Lemma \ref{eqrelation} that $X$ and
$X^{PQPQ}$ lie on different sides of $\ell$.

Now we show that (b) holds. Suppose $P\ne Q$. Then there exists a rimpoint $X$ not lying on the line
joining $P$ and $Q$ (by (A3)).  Then by (a), it is clear that $X\ne X^{PQPQ}$. Therefore, $PQPQ\ne
1$ and hence $PQ\ne QP$.
\end{proof}

\begin{lemma}\label{twotriangles}
Let $M$ and $N$ be two rimpoints, and let $\ell$ be the line spanned by them. Suppose $X$ and $Y$
are two distinct rimpoints lying on the same side of $\ell$. Then $XM$ intersects $YN$ or $XN$
intersects $YM$, in a finite point.
\end{lemma}

\begin{proof}
Let $P$ be a finite point on $MN$, and suppose $XM$ does not intersect $YN$. So $Y$ and $N$ are on
the same side of $XM$.  Now $P$ and $N$ are on the same side of $XM$ and hence $P$ and $Y$ are on
the same side of $XM$ (by Lemma \ref{eqrelation}).  So $PY$ does not meet $XM$. By \cite[Lemma 2.5]{DeBaggis:1948rt}, $PY$ meets $XN$ and so $P$ and $Y$ lie on different sides of $XN$. 
 Now $P$ and $M$ are on the same side of $XN$, and $P$ and
$Y$ are on different sides of $XN$. So by Lemma \ref{eqrelation}, $M$ and $Y$ are on different sides
of $XN$. Therefore, $XN$ intersects $YM$ in a finite point.
\end{proof}

Now we show that the relation $//$ from Definition \ref{defsep} satisfies (S3), (S4) and (S5).

\begin{theorem}\label{seprelation}
$//$ is a separation relation.
\end{theorem}

\begin{proof}
Let $A$, $B$, $C$, $D$, $X$ be distinct rimpoints. 
\begin{description}
\item[(S3) holds] The following argument is a direct analogue of \cite[Theorem 2.5]{Struve:2012gf}.
  Suppose $AB//CD$; that is, there exists a finite point $P$ on both $AB$ and $CD$.  For a proof by
  contradiction, suppose $AC$ and $BD$ intersect in a finite point $Q$.  Now $A^{P}=B$, $C^{P}=D$,
  $A^{Q}=C$ and $B^{Q}=D$, from which we see that $PQPQ$ fixes $A$, $B$, $C$ and $D$.  Suppose $P\ne
  Q$. Since $C\ne D$, we have that $A^Q\ne A^P$ and hence $A$ is not incident with $PQ$.  So by
  Lemma \ref{commutinginvolutions}, $A$ and $A^{PQPQ}$ lie on different sides of the
  line joining $A^{P}$ and $A^{PQ}$; which is a contradiction as $A=A^{PQPQ}$.  Therefore, $P=Q$ and
  $B=A^{P}=A^{Q}=C$; a contradiction. We have shown that $AB//CD$ implies that $AC//BD$ does not
  hold.

\item[(S4) holds] Suppose $AB//CD$ does not hold.  Then $A$ and $B$ lie on the same side of the
  line $\ell:=CD$, and so by Lemma \ref{twotriangles}, $AC$ intersects $BD$ or $AD$ intersects $BC$.
  Therefore, $AC//BD$ or $AD//BC$.

\item[(S5) holds] Suppose $AB//CD$ and $AC// BX$. In particular, we do not have $AB// CX$, by (S3).
  Now $\{A,B\}$ separates one of the pairs $\{C,D\}$, $\{D,X\}$, $\{X,C\}$ (namely $\{C,D\}$) and
  hence, by Corollary \ref{cor:Pasch}, two of the pairs.  Therefore, $AB// DX$.\qedhere
\end{description}
\end{proof}

\begin{lemma}\label{ACDF}
Consider two finite points $P$ and $Q$, on a line $\ell$, and suppose $A$ and $D$ are two rimpoints
on different sides of $\ell$.  Then $A D^{{P}{Q}}//A^{{P}{Q}} D$.
\end{lemma}

\begin{proof}
Let $X, X'$ be the rimpoints incident with $\ell$.  Then $XX' // A D$, and by (S3), we may suppose
without loss of generality that $XA // X' A^{PQ}$ (otherwise, reverse the roles of $X$ and $X'$).
So by (S5), we have\footnote{We can readily see how to apply (S5) via the following mapping of
  symbols $\left(\begin{smallmatrix} A&B&C&D&X\\ X&A&X'&A^{PQ}&D
\end{smallmatrix}\right)$.\smallskip
}
\[
XA // A^{PQ}D,\quad\text{and similarly,}\quad
XD // D^{PQ}A.
\]
Applying (S5) to both the statements `$XA// DA^{PQ} $ and $DX//AD^{PQ} $' and `$XD// AD^{PQ} $ and
$AX//DA^{PQ} $' yields\footnote{Again, it helps to know how we map symbols:
  $\left(\begin{smallmatrix} A&B&C&D&X\\ X&A&D&A^{PQ}&D^{PQ}
\end{smallmatrix}\right)$.\smallskip
}:
\[
AX // A^{PQ}D^{PQ} \quad\text{ and }\quad
DX // D^{PQ}A^{PQ}.
\]
Again, applying (S5) to `$AX// D^{PQ}A^{PQ}$ and $AD^{PQ}//XD$' yields\footnote{On
  this occasion, we use the mapping $\left(\begin{smallmatrix} A&B&C&D&X\\ A&D^{PQ}&X&D&A^{PQ}
\end{smallmatrix}\right)$.
}
$AD^{PQ} // A^{PQ}D$ as required.
\end{proof}

Inspired by \cite[Proposition 7.3]{Buekenhout:1966hc}, we have the following important result on half-turns.
We remind the reader that only axioms (A1) through (A7) are necessary for this result to hold.

\begin{theorem}\label{Buekenhout}
If $P$, $Q$, $R$ are collinear finite points belonging to a line $d$, then $PQR$ is an involution,
and moreover, $PQR={P'}$ for some point $P'$ on $d$.
\end{theorem}

\begin{proof}
Let $A$ be a rimpoint not incident with $PQ$. Let $B=A^{P}$, $C=B^{Q}$, $D=C^{R}$, $E=D^{P}$,
$F=E^{Q}$. Then $P=AB \cap DE$, $Q=BC \cap EF$.  By Lemma \ref{ACDF}, the lines $AD^{PQ}$ and
$DA^{PQ}$ meet, or in other words, $AF$ meets $CD$.  Then by (A7), $P$, $Q$ and $AF\cap CD$ are
collinear, and hence $AF\cap CD=PQ\cap CD=R$. So $A=F^{R}$, and $A$ and $D$ are interchanged by
$PQR$:
\[
A^{PQR}=B^{QR}=C^{R}=D,\quad D^{PQR}=E^{QR}=F^{R}=A.
\]
Since $A$ is arbitrary subject to not being on $PQ$, we have that $PQR$ is an involution.

Let $P':=PQ\cap AD$. We will show $PQR=P'$.  Take a rimpoint $Y$ not incident with $PQ$, and suppose
$Y$ lies on the opposite side of $PQ$ to $B$.  Let $Z:=Y^{{P'}}$, $Y':=Z^{P}$ and
$Z':=(Y')^{Q}$. Note that $B$ and $Y'$ lie on opposite sides of $PQ=PP'$ (by Lemma
\ref{interchangesides}).  By Lemma \ref{ACDF}, $BY$ meets $DY'$ (where $P$, $P'$, $Y'$, $B$ play the
roles of $P$, $Q$, $A$, $D$), since $A^{P'}=D$. Let $T:=BY \cap DY'$. By (A7) applied to the hexagon
$ABYZY'D$, we have that $P'=AD \cap YZ$, $P=AB \cap YZ'$ and $T=BY \cap DY'$ are collinear.  So $T$
is incident with $PQ$. Note that $D$ and $Y'$ lie on opposite sides of $PQ=TQ$ as $D$ lies on the
same side of $PQ$ as $B$ (by Lemma \ref{interchangesides}).  By Lemma \ref{ACDF}, $YZ'$ meets $CD$
(where $T$, $Q$, $Y$, $D$ play the roles of $P$, $Q$, $A$, $D$), since $Y^{TQ}=B^Q=C$ and
$D^{TQ}=(Y')^Q=Z'$.  By (A7) applied to the hexagon $CBYZ'Y'D$, we have that $T=BY \cap DY'$,
$Q=Y'Z' \cap BC$ and $YZ' \cap CD$ are collinear.  So $YZ' \cap CD$ is incident with $PQ$ and
therefore equals $CD \cap PQ=R$.  Thus $(Z')^{R}=Y$ and hence $Y$ is fixed by ${P'}PQR$.  Since $Y$
was arbitrary subject to being not incident with $PQ$ and lying on the opposite side to $PQ$ as $B$,
we have that ${P'}PQR$ fixes every rimpoint on one side of $PQ$. A similar argument shows that
${P'}PQR$ fixes every rimpoint on the other side of $PQ$ as well. Therefore, ${P'}PQR=1$ and so
$PQR={P'}$.
\end{proof}

\begin{theorem}\label{automorphism}
Each half-turn is an automorphism of $\mathcal{I}$.
\end{theorem}

\begin{proof}
Let $S$ be the set of half-turns as permutations of the set $\mathcal{O}$ of rimpoints.  We will
first show that the group $G$ generated by the half-turns is an \emph{$S$-group} as defined in
\cite[\S4]{Lingenberg:1979yu}: a group $G$ together with a distinguished generating subset $S$ of
the full set of involutions $J$, such that the following holds: if $a\ne b$ and $abx, aby, abz\in
J$, then $xyz\in S$.  Let $A$ and $B$ be distinct elements of $S$, and let $X,Y,Z\in S$ such that
$ABX$, $ABY$, $ABZ$ are involutions.  Now $ABX$ is an involution and so by Theorem
\ref{inv_then_collinear}, $A$, $B$, $X$ are collinear. Likewise, $Y$ and $Z$ lie on $AB$. Therefore,
$X$, $Y$ and $Z$ all lie on $AB$, and so by Theorem \ref{Buekenhout}, $XYZ\in S$.  Therefore, the
group $G$ generated by $S$ is an $S$-group.

Now given an $S$-group $(G,S)$ we can define an incidence structure $D(G,S)$ as follows: the points
are the elements of $S$, and a line is a set of points $S_{a,b}:=\{x: (abx)^2=1\}$ where $a$ and $b$
are two distinct elements of $S$.  By \cite[page 38]{Lingenberg:1979yu}, the conjugation action of
$G$ induces a subgroup of automorphisms of $D(G,S)$.

Let $\mathcal{P}$ be the set of points of $\mathcal{I}$ and let $\mathcal{L}$ be the lines of
$\mathcal{I}$.  For each point $X$, let $\beta_{\mathcal{P}}(X)$ be the half-turn about $X$. Note
that $\beta_{\mathcal{P}}$ is a bijection between $\mathcal{P}$ and $S$. So the conjugation action
of $G$ on $S$ induces a permutationally isomorphic action of $G$ on $\mathcal{P}$ via the map
$\beta_{\mathcal{P}}$: that is, for every point $X$ and point $P$, the image $X^P$ is the unique
point fixed by the half-turn $PXP$.  For lines, it is similar. Let $\ell\in\mathcal{L}$ and define
$\beta_{\mathcal{L}}( \ell)$ to be $S_{\beta_{\mathcal{P}}(M), \beta_{\mathcal{P}}(N)}$ where $M$
and $N$ are rimpoints of $\ell$. So again, the action of $G$ on pairs of elements of $S$ induces a
permutationally isomorphic action of $G$ on $\mathcal{L}$ whereby $\ell^P$ is the line spanned by
$M^P$ and $N^P$.  In $D(G,S)$, the half-turn $\beta_{\mathcal{P}}(X)$ about $X$ is incident with
$S_{\beta_{\mathcal{P}}(M), \beta_{\mathcal{P}}(N)}$ if and only if $(\beta_{\mathcal{P}}(M)
\beta_{\mathcal{P}}(N)\beta_{\mathcal{P}}(X))^2=1$.  By Theorem \ref{inv_then_collinear}, $MNX$ are
collinear and hence $X$ is incident with $MN$; and vice-versa (by Theorem \ref{Buekenhout}).
Therefore, $D(G,S)$ is isomorphic to $\mathcal{I}$ (as an incidence structure) and the result follows.
\end{proof}

\begin{corollary}\label{halfturn_inc}
The half-turn about $P$ fixes a line $\ell$ if and only if $P$ is incident with $\ell$.
\end{corollary}

\begin{proof}
Suppose $P$ fixes $\ell$, and let $Q$ be a point incident with $\ell$. 
Then by Theorem \ref{automorphism}, $Q^P$ is incident with $\ell^P$,
which is equal to $\ell$. Since $P$, $Q$ and $Q^P$ are collinear,
and $\ell$ is the unique line joining $Q$ and $Q^P$ (by (A1)),
it follows that $P$ is incident with $\ell$.
Conversely, if $P$ is incident with $\ell$, then the half-turn about $P$
fixes $P$ and maps $\ell$ to a line $\ell^P$ incident with $P$ (by Theorem \ref{automorphism}).
The rimpoints $L$ and $L'$ of $\ell$ are mapped to the rimpoints of $\ell^P$,
which must be interchanged by $P$. It follows that $\ell^P=\ell$.
\end{proof}

\begin{corollary}\label{preservesR}
Every half-turn preserves the separation relation $//$; that is, for every quadruple of distinct
rimpoints $W$, $X$, $Y$, $Z$, and for every finite point $P$, we have $WX//YZ$ if and only if
$W^PX^P//Y^PZ^P$.
\end{corollary}

A pair of transpositions $(A\, B)$ and $(C\, D)$ in the symmetric group on the rimpoints is
\emph{separating} if $\{A,B\}$ and $\{C,D\}$ are separating.

\begin{lemma}\label{fpfinv}
Let $g$ be a fixed-point-free involution in the symmetric group on the rimpoints such that each pair
of transpositions of $g$ are separating. Then $g$ is a half-turn.
\end{lemma}

\begin{proof}
Consider three pairs of transpositions of $g$: $(A\, B)$, $(C\, D)$, $(E\, F)$. Suppose, for a proof
by contradiction, that the lines $AB$, $CD$, $EF$ are not all incident with the same finite point.
Since every pair of these lines separate, we have three finite points $X:=AB\cap CD$, $Y:=CD\cap
EF$, $Z:=EF\cap AB$ fixed by $g$. Let us now consider the points $A$, $X$, $Z$, $B$ which lie on
$AB$. Then $A$, $C^X=D$, $C^Z$, $B$ are four distinct rimpoints.  Now $g$ is an automorphism of
$\mathcal{I}$ (by Theorem \ref{automorphism}), and so $C^{Zg}$ is equal to $D^Z$ (because $g$ maps
the line $CZ$ to $DZ$).  Since $D$ and $C^Z$ lie on the same side of $AB$, it follows that $AB$ does
not intersect $C^ZD$. So by (S4), $AC^Z// BD$ or $AD //C^ZB$.  Without loss of generality, we may
suppose that $AC^Z// BD$.  Now $g$ is an automorphism of $\mathcal{I}$, and so preserves the
separation relation (by Corollary \ref{preservesR}), and hence $A^{g}C^{Zg}//B^{g}D^g$. Computing
each image under $g$ yields $BD^Z// AC$. Now we apply $Z$ and reveal that $AD// BC^Z$.  However, by
assumption $AC^Z// BD$, which contradicts (S3).  Therefore, all lines corresponding to
transpositions of $g$ are concurrent in a single finite point $P$, and therefore, $g$ is a
half-turn.
\end{proof}

\begin{remark}\label{conjaxial0}
A consequence of Lemma \ref{fpfinv} (and Theorem \ref{Buekenhout}) is that if $P$ and $Q$ are finite points, then the map
$PQP$ is the half-turn about the point $Q^P$. Indeed, we have the following corollary of Lemma \ref{fpfinv}.
\end{remark}

\begin{corollary}\label{conjugationhalfturns}
Let $\tau$ be a permutation of $\mathcal{O}$ that preserves the separation relation $//$. Then for each finite point $P$,
the transformation $\tau^{-1}P\tau$ is a half-turn about the point $P^\tau$.
\end{corollary}

\begin{lemma}\label{preservesC}
Let $\cyclic$ be one of the two cyclic orders associated to the separation relation $//$. Then every
half-turn preserves $\cyclic$; that is, for every triple of distinct rimpoints $X$, $Y$, $Z$, and
for every finite point $P$, we have $\cyclic(X,Y,Z)$ if and only if $\cyclic(X^P,Y^P,Z^P)$.
\end{lemma}

\begin{proof}
Let $P$ be a finite point.  The half-turn about $P$ preserves the separation relation $//$, by
Corollary \ref{preservesR}.  So $P$ preserves or reverses the cyclic order $\cyclic$. Assume, for a
proof by contradiction, that $P$ reverses $\cyclic$.  Suppose we have $\cyclic(X,Y,Z)$ for three
distinct rimpoints $X$, $Y$, $Z$.  
If $Y=X^P$, then $Z$ and $Z^P$ lie on opposite sides of $XY$ and
we have by Remark \ref{remark:differentsides}, $\neg \cyclic(X,X',Z^P)$ and hence $\cyclic(X,Z^P,Y)$ (by (C2)).
However, after applying $P$, we have $\cyclic(Y,X,Z)$ which contradicts $\cyclic(X,Y,Z)$ and (C3). A similar
argument holds for $Y=Z^P$ or $X=Z^P$. So we can assume that $P$ does not lie on the lines $XY$,
$YZ$, $ZX$.  If $Y$ and $Z$ lie on the same side of $XX^P$, then $XZ//X^PY$ (because lying on the
same side of $XX^P$ means that $\cyclic(X,Y,X^P)$ and $\cyclic(X,Z,X^P)$, or $\cyclic(X^P,Y,X)$ and
$\cyclic(X^P,Z,X)$), and similarly, $Y^P$ and $Z^P$ lie on the same side of $XX^P$ and hence
$XZ^P//X^PY^P$.  However, $P$ preserves the separation relation $//$ (by Corollary \ref{preservesR})
and hence $X^PZ//XY$, which contradicts $XZ//X^PY$ (by (S3)).  Therefore, $Y$ and $Z$ lie on
opposite sides of $XX^P$.  In this case, $Y^P$ and $Z$ lie on the same side of $XX^P$.  So by
applying the previous argument with $\{Y^P,Z\}$ in the place of $\{Y,Z\}$ (the ordering dependent on
whether $\cyclic(X,Y^P,Z)$ or $\cyclic(X,Z,Y^P)$), we obtain again a contradiction.  Therefore, $P$
preserves $\cyclic$.
\end{proof}

\subsection{Adopting (A10) and enabling midpoints}

We now suppose that our incidence structure $\mathcal{I}$ satisfies (A10).
Recall from Axiom (A10) that given distinct finite points $P$, $Q$, there exist lines $\ell$ on $P$ and $m$ on $Q$ such that
two common parallels to $\ell$ and $m$ meet on $PQ$. We show now that this point is indeed
the `midpoint' of $\overline{PQ}$ in the language of half-turns.

\begin{lemma}[Midpoints exist]\label{midpointsexist}
Given two finite points $P$ and $Q$, there exists a finite point $M$ such that $PM=MQ$.
\end{lemma}

\begin{proof}
By (A10), there exist lines $\ell$ on $P$ and $m$ on $Q$ such
that if $X$, $Y$ are the rimpoints on $\ell$ and $X'$, $Y'$ are the rimpoints on $m$,
with $X$ and $X'$ lying on the same side of $PQ$, then $X'Y$ and $XY'$ meet
in a point $M$ incident with $PQ$. Then, $(X')^M=Y$ and $(Y')^M=X$ so $m^M=\ell$ (by Theorem \ref{automorphism}) and 
$(PQ)^M=PQ$, so $Q^M=(PQ \cap m)^M=PQ \cap \ell=P$  (again by Theorem \ref{automorphism}). Now the permutation $MQM$
is a half-turn by Lemma \ref{fpfinv}, and it fixes the point $Q^M=P$. Therefore,
$MQM=P$ and the result follows.
\end{proof}

\begin{lemma}[Midpoints are unique]\label{midpointsunique}
Let $P$ and $Q$ be finite points, and let $\ell$ and $m$ be lines through $P$ and $Q$ respectively
such that the intersection $M$ of the two common parallels to $\ell$ and $m$ is incident with $PQ$.
Then $PM=MQ$.
\end{lemma}

\begin{proof}
By Theorem \ref{Buekenhout}, there exists a finite point $N$ incident with $PQ$ such that the permutation $PMQ$ is equal
to the half-turn about $N$. Let $X$ be a rimpoint on $\ell$. Then the common parallels to $\ell$ and $m$ are $XX^{PMQ}$
and $X^PX^{PM}$, and so $N$ lies on both $\ell$ and $XX^{PMQ}$ (as $X^{N}=X^{PMQ}$).
Therefore, $M=N$ and $PM=MQ$.
\end{proof}

The theory for the remainder of this section was motivated by the late John Frankland Rigby's 1969 paper `Pascal ovals in projective planes' 
\cite{Rigby:1969ab}.

A nonidentity permutation $\alpha$ of the set $\mathcal{O}$ of rimpoints, preserving a cyclic order $\mathcal{C}$ compatible with the 
separation relation $//$, is an {\bf \axial} permutation if there exists
a line $\ell$ such that for all rimpoints $X$, $Y$ if $XY^\alpha$ meets $YX^\alpha$, it does so in a finite point incident with $\ell$. 
We then say that $\ell$ is the {\bf axis} of $\alpha$.  

\begin{lemma}\label{axialfixpoints}
The set of fixed rimpoints of an \axial\ mapping $\alpha$ with axis $\ell$ is the set of rimpoints on $\ell$. 
Hence the axis is unique, and an \axial\ mapping has exactly two fixed points.
\end{lemma}

\begin{proof}
Let $X$ be a fixed rimpoint, and suppose $Y$ is a rimpoint that is not fixed by $\alpha$. Then $X=XY^\alpha\cap YX^\alpha$ lies on $\ell$. 
Conversely, suppose that $A$ is a rimpoint on $\ell$, not fixed by $\alpha$, and let $Y$ be a point not incident with $\ell$
such that $AY^\alpha$ meets $A^\alpha Y$ in a finite point $Q$.
Then $Q$ does not lie on $\ell$; a contradiction. So the fixed rimpoints of $\alpha$ are exactly those incident with $\ell$.
\end{proof}

Since $\alpha$ preserves $\mathcal{C}$ and fixes the rimpoints of $\ell$, it follows that if $XY^\alpha$ meets $YX^\alpha$, 
then $X$ and $X^\alpha$ lie on the same side of $\ell$, and $X$ and $Y$ lie on opposite sides of $\ell$.

%\begin{lemma}\label{agreeingaxials}
%If two \axial\ mappings with the same axis $\ell$ agree on a rimpoint not incident with $\ell$, then they are equal.
%\end{lemma}
%
%\begin{proof}
%Let $\alpha$ and $\beta$ be the two \axial\ mappings in the hypothesis, and suppose $X$ is a rimpoint not incident with $\ell$
%such that $X^\alpha=X^\beta$. Let $Y$ be a rimpoint, on the opposite to $\ell$ as $X^\alpha$. 
%
%
%$XY^\alpha$ meets $YX^\alpha$, it does so in a finite point incident with $\ell$. 
%
%TBC
%\end{proof}

\begin{lemma}\label{PQaxial}
Let $P$, $Q$ be distinct finite points. Then $PQ$ is an \axial\ mapping with axis the line $\ell$ joining $P$ and $Q$.
\end{lemma}

\begin{proof}
Certainly, $PQ$ fixes the rimpoints on $\ell$. If $X$ and $Y$ are rimpoints not on $\ell$, apply (A7) 
to the hexagon with vertices 
\[
X, X^P, X^{PQ}, Y, Y^P, Y^{PQ}.
\]
We know that $XX^P$ and $YY^P$ meet at $P$, and $X^PX^{PQ}$ and $Y^PY^{PQ}$ meet at $Q$, so if $XY^{PQ}$ meets $YX^{PQ}$, it does so on $\ell=PQ$.
\end{proof}

%
%\begin{lemma}\label{axialfixessides}
%An \axial\ map fixes the sides of its axis.
%\end{lemma}
%
%\begin{proof}
%This follows directly from Lemma \ref{PQaxial} and Lemma \ref{preservesC}. 
%\end{proof}

\begin{lemma}\label{axial2halfturns}\label{productaxials}\leavevmode
\begin{enumerate}[(i)]
\item
If $\alpha$ is an \axial\ map with axis $\ell$, and if $P$ is a finite point on $\ell$, then there exists a unique finite point $Q$ on 
$\ell$ such that $\alpha=PQ$.

\item
The composition of two \axial\ mappings with the same axis is an \axial\ map (with the same axis), or the identity mapping.
\end{enumerate}
\end{lemma}

\begin{proof}
Let $X$ be a rimpoint not on $\ell$. Then $X^P$ and $X^\alpha$ are on opposite sides of $\ell$ (by Lemma \ref{interchangesides})
since an \axial\ map fixes the sides of its axis (by Lemma \ref{PQaxial} and Lemma \ref{preservesC}).
So the line $X^PX^\alpha$ meets $\ell$ in a point $Q$. Let $Y$ be a rimpoint not incident with $\ell$,
lying on the other side of $\ell$ to $X$. So $XY^\alpha$ meets $YX^\alpha$ in a finite point $T$, and by definition of an \axial\ map, $T$
lies on $\ell$. Consider the hexagon $XX^PX^\alpha YY^PY^\alpha$. Then the diagonal points are 
$XX^P\cap YY^P=P$, $XY^\alpha\cap YX^\alpha=T$, and $X^PX^\alpha\cap Y^PY^\alpha$.
Now $X^P$ lies on the opposite side of $\ell$ to $X$, and $X^\alpha$ lies on the same side of $\ell$ as $X$.
Therefore, $X^PX^\alpha\cap Y^PY^\alpha$ exists and lies on $\ell$,
and hence $X^PX^\alpha\cap Y^PY^\alpha=Q$. Therefore, $Y^\alpha=Y^{PQ}$. So we have shown that $\alpha$ is equal to $PQ$
on the rimpoints lying on the side of $\ell$ opposite to $X$. By a similar argument, and consideration of a rimpoint $X'$ on the opposite
of $\ell$ to $X$, there exists a finite point $Q'$ incident with $\ell$ such that $\alpha$ is equal to $PQ'$ on the 
 rimpoints lying on the same side of $\ell$ as $X$. However, we also know, by definition of $Q$, that $X^\alpha=X^{PQ}$, 
 and hence $QQ'$ fixes $X^P$. So by Lemma \ref{axialfixpoints}, $QQ'$ is the identity permutation and $Q=Q'$.
 Therefore, $\alpha=PQ$ and (i) holds.

Suppose $\alpha$ and $\beta$ are two \axial\ maps with axis $XY$ (where $X$ and $Y$ are rimpoints). 
Let $P$ be a point incident with $XY$. By Lemma \ref{axial2halfturns}, there exist points $Q$ and $R$ incident
with $XY$ such that $\alpha=PQ$ and $\beta=PR$. 
Hence, $\alpha\beta=PQPR=Q^PR$ which is the identity if $R=Q^P$, or otherwise,
an \axial\ map, by Remark \ref{conjaxial0} and Lemma \ref{PQaxial}. Therefore, (ii) holds.
\end{proof}

\begin{remark}\label{conjaxial}
It is a simple consequence of Lemma \ref{axial2halfturns} that if $\alpha$ is an \axial\ map with axis $\ell$ and $A$ is a finite point,
then the map $A\alpha A$ is an \axial\ map with axis $\ell^A$. In fact, if $\alpha=PQ$, then $A\alpha A=(APA)(AQA)=P^AQ^A$,
by Theorem \ref{automorphism} and Remark \ref{conjaxial0}.
\end{remark}

\begin{lemma}\label{noorder2}
An \axial\ map cannot interchange two rimpoints.
\end{lemma}

\begin{proof}
Let $\alpha$ be an \axial\ map and suppose it interchanges two rimpoints $X$ and $Y$. By Lemma \ref{axial2halfturns},
there exist two finite points incident with the axis of $\alpha$ such that $\alpha=MN$.
Then $\{X,X^{M}\}$ separates $\{Y,Y^{M}\}$ and hence $XY^{M}$ does not meet $YX^{M}$
(by (S3)); 
a contradiction as $N$ is the point of intersection of these two lines.
\end{proof}

%\begin{lemma}\label{productaxials}
%The composition of two \axial\ mappings with the same axis is an \axial\ map (with the same axis), or the identity mapping.
%\end{lemma}
%
%\begin{proof}
%Suppose $\alpha$ and $\beta$ are two \axial\ maps with axis $XY$ (where $X$ and $Y$ are rimpoints). 
%Let $P$ be a point incident with $XY$. By Lemma \ref{axial2halfturns}, there exist points $Q$ and $R$ incident
%with $XY$ such that $\alpha=PQ$ and $\beta=PR$. 
%Hence, $\alpha\beta=PQPR=Q^PR$ which is the identity if $R=Q^P$, or otherwise,
%an \axial\ map, by Remark \ref{conjaxial0} and Lemma \ref{PQaxial}. 
%\end{proof}

\begin{lemma}\label{uniqueaxial}
Given a line $\ell$ and rimpoints $U$, $V$ on the same side of $\ell$ there is a unique \axial\ mapping $\alpha$ with axis $\ell$ such that $U^\alpha=V$.
\end{lemma}

\begin{proof}
Let $P$ be a finite point on $\ell$ and let $W:=U^P$. Then $WV$ intersects $\ell$, as $V$, $W$ are on opposite sides of $\ell$; let 
$Q:=WV \cap \ell$. Then $\alpha:=PQ$ is an \axial\ map (by Lemma \ref{PQaxial}) taking $U$ to $V$. Any other such mapping can be written as $PQ'$, for some $Q'$ on $\ell$ (by Lemma \ref{axial2halfturns}). Since $U^{PQ'}=V$, it follows that $W^{Q'}=V$ and hence $Q=Q'$.
\end{proof}

%\textcolor{blue}{Don't use this}
%
%\begin{lemma}\label{axial3stab}
%If $X$, $Y$, $Z$ and $X'$, $Y'$, $Z'$ are triples of distinct rimpoints, then there is at most one \axial\
%element $g$ of $G$ such that $X^g=X'$, $Y^g=Y'$ and $Z^g=Z'$. 
%\end{lemma}
%
%\begin{proof}
%We may suppose that $Z \ne Z'$.
%If there is an axial mapping with this property, then $XZ'$ meets $X'Z$ in a finite point,
%and $YZ'$ meets $Y'Z$ in a finite point (and the axis is the line $\ell$ joining these points); so suppose that this is the case. By the immediately preceding observation (Lemma \ref{uniqueaxial}), there is a unique element of $G$ with axis $\ell$ taking $Z$ to $Z'$, 
%provided they are on the same side of $\ell$.
%\end{proof}

\begin{lemma}\label{axialcyclic}
Let $\mathcal{C}$ be a cyclic order compatible with $//$, and let $\alpha$ be an \axial\ mapping with axis $\ell$.
Then for every rimpoint $X$ not incident with $\ell$, we have either
$\mathcal{C}(X^\alpha, X, X^{\alpha^{-1}})$ or $\mathcal{C}(X^{\alpha^{-1}},X,X^\alpha)$.
\end{lemma}

\begin{proof}
By Lemma \ref{PQaxial}, there exist finite point $P$ and $Q$ incident with $\ell$ such that $\alpha=PQ$.
Let $X_1$ and $X_2$ be the rimpoints of $\ell$, such that we have the ordering $X_1\,P\,Q\,X_2$.
Let $X$ be a rimpoint not incident with $\ell$. If (without loss of generality) $\cyclic(X_1,X,X_2)$, then we have the ordering,
$X_2\,X^Q\,X^P\,X_1$ (via projection). Therefore, $\cyclic(X^Q,X^P,X_1)$
and hence $\cyclic(X,X^{PQ},X_2)$ (by Lemma \ref{preservesC}). By a similar argument,
$\cyclic(X_1,X,X^{QP},X)$. Therefore, $\mathcal{C}(X^{\alpha^{-1}},X,X^\alpha)$.
If $P$ and $Q$ are situated such that we have the ordering $X_1\,Q\,P\,X_2$, then we have the other case.
\end{proof}

%\begin{theorem}\label{atmost2fixed}
%If $P$, $Q$, $R$, $S$ are finite points, then either $PQRS=1$ or $PQRS$ has at most two fixed rimpoints.
%\end{theorem}
%
%\begin{proof}
%Consider first the product $MN$ of two half-turns, and suppose $MN$ fixes three rimpoints.
%Then $M$ and $N$ have three transpositions in common,
%$(A\, A')$, $(B\, B')$ and $(C\,C')$, say, with at most two coinciding. Since
%a half-turn is determined by two of its transpositions, it follows that $M=N$.
%So we have dealt with the case that $P=Q$ or $Q=R$ or $R=S$, so suppose that none of these equalities hold. 
%Suppose $PQRS$ fixes three distinct rimpoints $X$, $Y$, $Z$. Then $(X,Y,Z)^{PQ}=(X,Y,Z)^{SR}$, and $PQ$ and $SR$ are axial, so $PQ=SR$ (by Lemma \ref{axial3stab}), giving $PQRS=1$.
%\end{proof}

\begin{lemma}\label{axialcommutes}
Two \axial\ maps with the same axis commute.
\end{lemma}

\begin{proof}
Let $\alpha$ and $\alpha'$ be two \axial\ maps with common axis $\ell$. By Lemma \ref{axial2halfturns}, there exist three finite points $P$, $P'$, $Q$ incident with $\ell$
such that $\alpha=PQ$ and $\alpha'=QP'$. By Theorem \ref{Buekenhout}, there exists a finite point $Q'$ incident with $\ell$ such that
$P'PQ=Q'$ and hence $P'PQ = (P'PQ)^{-1}=QPP'$. Therefore,
\[
\alpha'\alpha = QP'PQ = Q(P'PQ)=Q(QPP')=PP'=\alpha\alpha'
\]
and we have proved our result.
\end{proof}

The following two lemmas will be used later to define \emph{perpendicularity} of lines of $\mathcal{I}$.
Later we will see that Lemma \ref{uniqueaxialsquare} allows us to easily express the unique common
perpendicular to two lines (see Lemma \ref{commonperp}).

\begin{lemma}\label{squares}
Every \axial\ mapping is the square of an \axial\ mapping (with the same axis).
\end{lemma}

\begin{proof}
Let $\alpha$ be an \axial\ mapping with axis $\ell$. By Lemma \ref{axial2halfturns}, 
there exist finite points $P$ and $Q$ incident with $\ell$ such that $\alpha=PQ$.
By Lemma \ref{midpointsexist}, there exists a finite point $M$ incident with $PQ$ such that
$P^M=Q$ (i.e., $M$ is the \emph{midpoint} of the segment $\overline{PQ}$). 
Then $(PM)^2=PMPM=PP^M=PQ=\alpha$, and we know that $PM$ is an \axial\ map by Lemma \ref{PQaxial},
with axis $\ell$.
Therefore, $\alpha$ is the square of an \axial\ mapping.
\end{proof}

\begin{lemma}\label{uniqueaxialsquare}
Let $X_1$, $X_2$, $X_3$, $X_4$ be rimpoints such that $X_1X_2$ and $X_3X_4$ are non-intersecting non-parallel lines. Then
there is a unique \axial\ map $g$ with axis $X_1X_2$ such that $X_4^{g^2}=X_3$.
\end{lemma}

\begin{proof}
Let $\ell$ be the line $X_1X_2$.
By Lemma \ref{uniqueaxial}, there exists an \axial\ map $\alpha$ with axis $\ell$ such that $X_4^\alpha=X_3$.
By Lemma \ref{squares}, there is an \axial\ map $g$ with axis $\ell$ such that $\alpha=g^2$.
So we know that an \axial\ map $g$ with axis $\ell$ exists such that $X_4^{g^2}=X_3$.
By Lemma \ref{axialcommutes} and Lemma \ref{productaxials}, the \axial\ mappings with axis $\ell$ form an Abelian group $A(\ell)$,
and so the squaring map $x\mapsto x^2$ is a homomorphism on $A(\ell)$.
There are no elements of order $2$ in $A(\ell)$ (by Lemma \ref{noorder2}) and so
the squaring map is injective. Hence $g$ is the unique
element of $A(\ell)$ such that $X_4^{g^2}=X_3$.
\end{proof}

%%%%%%%%%%%%%%%%%%%%%%%%%
%
% Line-reflections
%
%%%%%%%%%%%%%%%%%%%%%%%%%

\section{Line-reflections and perpendicularity}\label{section:definedperp}

Throughout this section, we will suppose that $\mathcal{I}$ is an incidence structure satisfying
axioms (A1) through (A7), plus (A10) and (A11). 
Suppose $\ell$ and $m$ are non-parallel lines, and suppose $\ell$ has rimpoints $\{L_1,L_2\}$
and $m$ has rimpoints $\{M_1,M_2\}$ such that $L_1M_2//L_2M_1$. We define $\ell\bowtie m$ to be the finite
point $L_1M_2\cap L_2M_1$. In this notation, Axiom (A11) states that if $X$ is a rimpoint not on a line $\ell$,
then there is a unique point $R$ on $\ell$ such that for every finite point $S$ on $\ell$,
 $\ell \bowtie X^{SR}X^{RS}$ is incident with $XR$.

%We restate Axiom (A11) in terms of axial maps.

\begin{lemma}\label{commonaxis}
Let $\alpha$ and $\beta$ be \axial\ maps with a common axis $\ell$. Let $A$ be a rimpoint not incident 
with $\ell$. Then $A$, $\ell \bowtie A^\alpha A^{\alpha^{-1}}$, $\ell \bowtie A^\beta A^{\beta^{-1}}$
are collinear.
\end{lemma}

\begin{proof}
By (A11), there is a unique finite point $R$ on $\ell$ such that for every finite point $S$ on $\ell$,
$\ell \bowtie A^{SR}A^{RS}$ is incident with $AR$. By Lemma \ref{axial2halfturns}, 
there exist finite points $M$ and $N$ on $\ell$ such that $\alpha=RM$ and $\beta=RN$.
Therefore, $\ell \bowtie A^\alpha A^{\alpha^{-1}}=\ell \bowtie A^{RM}A^{MR}$ 
and $\ell\bowtie A^{RM}A^{MR}$ is incident with $AR$. Likewise, 
$\ell \bowtie A^\beta A^{\beta^{-1}}$ is incident with $AR$, and so 
$A$, $\ell \bowtie A^\alpha A^{\alpha^{-1}}$, $\ell \bowtie A^\beta A^{\beta^{-1}}$
all lie on $AR$.
\end{proof}

\begin{definition}\label{defperp}
Let $\ell$ be a line and let $A$ be a rimpoint not incident with $\ell$. Let $X_1$ and $X_2$ be the rimpoints on $\ell$ and let $\alpha$ be an 
\axial\ map with axis $\ell$. Let $O:=\ell\bowtie A^\alpha A^{\alpha^{-1}}$.
Then $AA^O$ is the \textbf{perpendicular line} to $\ell$ incident with $A$.
\end{definition}

Axiom (A11) (see Lemma \ref{commonaxis}) 
implies that the definition (in Definition \ref{defperp}) of the perpendicular line to $\ell$ incident with $A$  
is independent of the choice of $\alpha$, and hence well-defined.
Therefore, given a line $\ell$, there is a function $\sigma_\ell$ on the rimpoints defined as follows:
$\sigma_\ell$ fixes the rimpoints of $\ell$, and for a rimpoint $A$ not incident with $\ell$, 
we define $A^{\sigma_\ell}$ to be the rimpoint $A^O$ given in Definition \ref{defperp}.
By Lemma \ref{refinv} (below), $\sigma_\ell$ is a permutation of the rimpoints and we will refer to this map as the \textbf{reflection in the line $\ell$}.

\begin{lemma}\label{perpconc}
Let $\ell$ be a line and let $A$ be a rimpoint not incident with $\ell$. Then the line $AA^{\sigma_\ell}$ meets $\ell$ in a finite point.
Therefore, $\sigma_\ell$ interchanges the sides of $\ell$.
\end{lemma}

\begin{proof}
Let $\alpha$ be an \axial\ map with axis $\ell$, and let $\{X_1,X_2\}$ be the rimpoints of $\ell$. Let $\mathcal{C}$ be a cyclic order
on the rimpoints compatible with the separation relation $//$ such that $\mathcal{C}(X_1,A,X_2)$. 
Now $A^\alpha$ and $A^{\alpha^{-1}}$ lie on the same side of $\ell$ as $A$ (by Lemma \ref{interchangesides}), and either
$\mathcal{C}(A^\alpha, A, A^{\alpha^{-1}})$ or $\mathcal{C}(A^{\alpha^{-1}},A,A^\alpha)$ (by Lemma \ref{axialcyclic}).
Suppose, without loss of generality, that $\mathcal{C}(X_1,A^\alpha,A)$ and $\mathcal{C}(A,A^{\alpha^{-1}},X_2)$. 
Let $O:=\ell\bowtie A^{\alpha}A^{\alpha^{-1}}$. In particular, from our ordering, we have $X_1^O=A^{\alpha^{-1}}$ and $X_2^O=A^\alpha$.
Now the half-turn $O$ preverses the cyclic order $\mathcal{C}$, and hence
$\mathcal{C}(X_1^O,A^{\alpha O}, A^O)$ and $\mathcal{C}(A^O,A^{\alpha^{-1}O},X_2^O)$. 
So $\mathcal{C}(A^{\alpha^{-1}},X_2,A^O)$ and $\mathcal{C}(A^O,X_1,A^\alpha)$,
and it follows that $\mathcal{C}(X_2,A^O,X_1)$. In particular, $A^{\sigma_\ell}=A^O$ lies on the opposite side of $\ell$ as $A$.
\end{proof}

\begin{lemma}\label{refinv}
For any line $\ell$, we have $\sigma_\ell^2=1$.
\end{lemma}

\begin{proof}
Let $A$ be a rimpoint not incident with $\ell$, and let $\alpha$ be an \axial\ map with axis $\ell$.
By definition, $A^{\sigma_{\ell}}=A^O$ where $O=\ell\bowtie A^\alpha A^{\alpha^{-1}}$.
Let $O' :=\ell\bowtie (A^O)^\alpha (A^O)^{\alpha^{-1}}$. 
By Lemma \ref{perpconc}, $\ell$ meets $AO$ in a finite point $R$.
Write $\alpha=PQ$ for two finite points $P$ and $Q$ incident with $\ell$ (see Lemma \ref{axial2halfturns}).
Since $P$, $Q$, $R$ are collinear, we have from Theorem \ref{Buekenhout} that $PQR=RQP$ and $QPR=RPQ$. So
\begin{align*}
\left((A^O)^\alpha\right)^R&=A^{O(PQR)}=A^{O(RQP)}=(A^{OR})^{QP}=A^{QP}=A^{\alpha^{-1}},\\
\left((A^O)^{\alpha^{-1}}\right)^R&=A^{O(QPR)}=A^{O(RPQ)}=(A^{OR})^{PQ}=A^{PQ}=A^{\alpha}.
\end{align*}
Therefore, since half-turns are automorphisms by Theorem \ref{automorphism}, we have
\[
(O')^R=(\ell \bowtie (A^O)^\alpha (A^O)^{\alpha^{-1}})^R=\ell\bowtie A^{\alpha^{-1}}A^\alpha=O
\]
as $\ell^R=\ell$. So $O'$ lies on the line $RO=AO$ and therefore, 
$\left(A^{\sigma_{\ell}}\right)^{\sigma_\ell}=(A^{\sigma_\ell})^{O'}=A^{OO'}=A$.
Thus we have shown that $\sigma_\ell^2$ fixes every rimpoint, and so is the identity.
\end{proof}

\begin{lemma}\label{perpsym}
Let $\ell$ and $m$ be two lines meeting in a finite point. Then the following are equivalent:
\begin{enumerate}[(a)]
\item $\sigma_\ell$ interchanges the rimpoints of $m$;
\item $\sigma_m$ interchanges the rimpoints of $\ell$;
\end{enumerate}
\end{lemma}

\begin{proof}
Suppose $\sigma_\ell$ interchanges the rimpoints $\{M_1,M_2\}$ of $m$.
Let $\alpha$ be an \axial\ map with axis $\ell$.
Let $\{X_1,X_2\}$ be the rimpoints of $\ell$, written in such a way that $\{X_1,M_1^{\alpha^{-1}}\}$ separates $\{X_2,M_1^{\alpha}\}$.
Let $T:=X_1M_1^{\alpha^{-1}}\cap X_2M_1^\alpha$ and $U:=\ell\cap m$. By assumption, $M_1^{\sigma_\ell}=M_2$
and hence $M_2=M_1^T$. Therefore, $T$ is incident with $m$ and
$m=M_1T$ is the perpendicular line to $\ell$ through $M_1$.
Let $\beta:=UT$. Then 
\begin{align*}
X_1^{\beta}&=(X_1^U)^T=X_2^T=M_1^\alpha,\\
X_1^{\beta^{-1}}&=(X_1^T)^U=M_1^{\alpha^{-1}U},
\end{align*}
 and so $X_1P$ is the perpendicular to $m$ through $X_1$ where $P:=M_1^{\alpha^{-1}U}M_1 \cap M_2M_1^{\alpha}$.
We will show that $P$ is incident with $\ell$. 
Now $PUTP$ is axial with axis $m^P=M_1^{\alpha^{-1}U}M_1^\alpha$ (see Remark \ref{conjaxial}), and 
\[
M_1^{(PUTP)^2}=M_1^{PUTUTP}=(M_1^{\alpha^{-1}})^{TUTP}=X_1^{UTP}=X_2^{TP}=M_1^{\alpha P}=M_2.
\]
So by Lemma \ref{uniqueaxial}, $PUTP=T'U'$ where $U'=\ell\cap M_1^{\alpha}M_1^{\alpha^{-1} U}$ and $T'=X_1M_1\cap X_2M_2^{U'}$.
(Note: $M_1^{(T'U')^2}=M_1^{T'U'T'U'}=X_1^{U'T'U'}=X_2^{T'U'}=M_2$). Therefore,
\[
X_1^P=(M_1^{\alpha^{-1}})^{TP}=(M_1^{\alpha^{-1} U})^{UTP}=M_1^{PUTP}=M_1^{T'U'}=X_1^{U'}=X_2.
\]
Hence $P$ lies on $\ell$ and $X_1P=\ell$. It follows that $\sigma_m$ interchanges the rimpoints of $\ell$.
Interchanging the roles of $\ell$ and $m$ in this argument shows that the converse also holds.
\end{proof}

Lemma \ref{perpsym} allows us to define perpendicularity of lines. We say that a line $\ell$ and a line $m$ are \textbf{perpendicular}, and we write $\ell\perp m$, if $\ell$ and $m$ are concurrent and 
$\sigma_\ell$ interchanges the rimpoints of $m$. So Lemma \ref{perpsym} implies that $\perp$ is a symmetric relation on lines.
%Also, we have already established via Lemma \ref{refinv} that reflections are involutions and so given
%a line $\ell$ and a rimpoint $P$ not incident with $\ell$, there is a unique line through $P$ perpendicular to $\ell$.
Recall from Theorem \ref{automorphism} that half-turns are automorphisms of $\mathcal{I}$. They also preserve
the perpendicularity relation (Definition \ref{defperp}).

\begin{lemma}\label{halfturnperp}
Half-turns preserve perpendicularity. That is, if $\ell\perp m$ and $P$ is a finite point, then $\ell^P\perp m^{P}$.
\end{lemma}

\begin{proof}
Suppose $\ell\perp m$ and let $P$ be a finite point. Let $M$ and $M'$ be the rimpoints of $m$, and
let $\alpha$ be an axial map with axis $\ell$. Then, $m=MO$ where $O=\ell\bowtie M^\alpha M^{\alpha^{-1}}$.
Let $\beta:=P\alpha P$.
By Remark \ref{conjaxial}, $\beta$ is axial with axis $\ell^P$.
Moreover, by Lemma \ref{automorphism}, we have
\begin{align*}
O^P&=\ell^P\bowtie M^{\alpha P} M^{\alpha^{-1} P}\\
&=\ell^P\bowtie M^{P\beta} M^{P\beta^{-1} }\\
&=\ell^P\bowtie (M^P)^\beta (M^P)^{\beta^{-1}}
\end{align*}
and hence $M^PO^P$ is perpendicular to $\ell^P$. That is, $m^P\perp \ell^P$.
\end{proof}

\begin{lemma}\label{commonperp}
Suppose $\ell$ and $m$ are two non-parallel non-intersecting lines. Then there exists a unique line $n$ perpendicular to both
$\ell$ and $m$. Moreover, there are unique \axial\ maps $g$ and $h$ such that (i) $g$ has axis $\ell$, (ii) $h$ has axis $m$, (iii) $X_4^{g^2}=X_3$, (iv) $X_2^{h^2}=X_1$. The line $n$ is equal to $X_4^gX_2^h$ and is the \textbf{common perpendicular} of $\ell$ and $m$.
\end{lemma}

\begin{proof}
Let $\{X_1,X_2\}$ be the rimpoints of $\ell$ and let $\{X_3,X_4\}$ be the rimpoints of $m$ such that $X_1X_4//X_2X_3$.
By Lemma \ref{uniqueaxialsquare}, there is a unique \axial\ map $g$ with axis $\ell$ such that $g^2$ maps $X_4$ to $X_3$.
Likewise, there is a unique \axial\ map $h$ with axis $m$ such that $h^2$ maps $X_2$ to $X_1$.
Let $A:=X_4^g$. Then
\[
A^g=X_3\text{ and } A^{g^{-1}}=X_4.
\]
Let $O:=X_1X_4\cap X_2X_4$. By definition, the perpendicular line to $\ell$ through $A$ is $n:=AO$.
Now consider the map $Og^{-1}O$. First, by Lemma \ref{PQaxial}, $Og^{-1}O$ is an \axial\ map as it is the product of two half-turns: if $g=PQ$,
then $Og^{-1}O=Q^OP^O$ (by Lemma \ref{fpfinv}), and so $Og^{-1}O$ has axis $m$. Also,
\[
X_2^{(Og^{-1}O)^2}=X_2^{Og^{-2}O}=X_3^{g^{-2}O}=X_4^O=X_1.
\]
Therefore, by Lemma \ref{uniqueaxialsquare}, $h=Og^{-1}O$.
Let $B:=X_2^h$. Then
\[
B^h=X_1\text{ and } B^{h^{-1}}=X_2,
\]
and so the perpendicular line to $m$ through $B$ is $BO$. On the other hand,
\[
X_2^h=X_2^{Og^{-1}O}=X_3^{g^{-1}O}=A^O,
\]
and therefore $B=A^O$ and $BO=n$. So we have shown that $n$ is perpendicular to $\ell$ and $m$.

It remains to prove uniqueness of the line $n$. 
Let $n'$ be another line perpendicular to both $\ell$ and $m$. (The following argument is similar
to that given in \cite[Theorem 3.8]{Abbott:1943aa}).
Let the rimpoints of $n$ and $n'$ be $U$ and $V$, and $U'$ and $V'$, respectively, arranged so
that $U$ and $U'$ lie on the same side of $\ell$ and $m$. So there exist finite points $Q_1$ and $Q_2$
such that $Q_1=\ell\cap U'V$ and let $Q_2=m\cap U'V$.
 Let $V_1'$ and $V_2'$ be the other rimpoints on $UQ_1$ 
and $UQ_2$, respectively. Since $\ell$ is perpendicular to $n$, and the half-turn about $Q_1$ maps $V$ to $U_1$ and $U$ to $V_1'$,
by Lemma \ref{halfturnperp}, $U_1'V_1'$ is perpendicular to $\ell$ too. However, the perpendicular line to $\ell$ through
a rimpoint not incident with $\ell$ is unique and hence $U_1'V_1'=n'=U'V'$. Therefore, $V'=V_1'$. Similarly,
using the fact that $m$ is perpendicular to $n$, we have via the half-turn about $Q_2'$ that $U'V_2'$ is the perpendicular line
to $m$ through $U'$, and so $V'=V_2'$. Thus $Q_1=U'V\cap UV'=Q_2$ is on $\ell$ and $m$; contrary to the fact that $\ell$
and $m$ do not intersect. Therefore, $n$ is the unique line perpendicular to $\ell$ and $m$.
\end{proof}

We will denote the unique common perpendicular to non-parallel non-intersecting lines $\ell$ and $m$ by
$\bot(\ell,m)$.

\begin{lemma}\label{diagonalpointpolar}
Let $\ell$ and $m$ be non-intersecting non-parallel lines. % and let $u$ and $v$ be the two common parallel lines to $\ell$ and $m$.
Then $\ell\bowtie m$ %$u\cap v$ 
is incident with $\bot(\ell,m)$. Moreover, $\ell \bowtie m$ is the midpoint of the segment joining $\ell\cap \bot(\ell,m)$ and $m\cap \bot(\ell,m)$.
\end{lemma}

\begin{proof}
Let $X$ and $X'$ be the rimpoints on $\ell$, let $Y$ and $Y'$ be the rimpoints on $m$, labelled so that $u=XY'$ 
meets $v=X'Y$ in a finite point $P$. Then the half-turn about $P$ interchanges $X$ and $Y'$, and also $Y$ and $X'$; and hence interchanges $\ell=XX'$ and $m=Y'Y$
(as half-turns are automorphisms by Theorem \ref{automorphism}). Let $n:=\bot(\ell,m)$.  Then (by Lemma \ref{commonperp})
there are unique \axial\ maps $g$ and $h$ such that (i) $g$ has axis $\ell$, (ii) $h$ has axis $m$, (iii) $Y^{g^2}=Y'$, (iv) $X^{h^2}=X'$.
By Lemma \ref{axial2halfturns} and Theorem \ref{automorphism}, $PgP$ is \axial\ with axis $m$. Now
\[
(X')^{(PgP)^2}=(X')^{PgPPgP}=(X')^{Pg^2P}=Y^{g^2P}=(Y')^P=X
\]
and so by Lemma \ref{uniqueaxialsquare}, $PgP=h^{-1}$. Therefore, $Y^{gP}=(X')^{PgP}=(X')^{h^{-1}}=X^h$ and hence 
$P$ fixes $n=Y^gX^h$. So by Corollary \ref{halfturn_inc}, $P$ is incident with $n$. Moreover, (by Theorem \ref{automorphism}),
$(\ell\cap n)^P=\ell^P\cap n^P=m\cap n$ and hence $P$ is the midpoint of the segment joining $\ell\cap n$ and $m\cap n$.
\end{proof}

\begin{lemma}\label{parallel}
Two distinct lines $\ell$ and $m$ are parallel if and only if they are not concurrent (in a finite point) nor share
a common perpendicular. 
\end{lemma}

\begin{proof}
If $\ell$ and $m$ are not concurrent in a finite point, nor parallel, then they share a common perpendicular by Lemma \ref{commonperp}.
Conversely, suppose $\ell$ and $m$ are parallel; that is, $\ell$ and $m$ share a rimpoint $X$.
Clearly, $\ell$ and $m$ cannot meet in a finite point $P$, say, since otherwise, $\ell=XP=m$.
So suppose a line $n$ is perpendicular to $\ell$ and $m$. Let $g$ be an \axial\ map with axis $\ell$.
Then there is a rimpoint $A$ incident with $n$, such that if $O:=\ell \bowtie A^gA^{g^{-1}}$, then
$n=AO$. Now $\sigma_n$ is a permutation of the rimpoints, by Lemma \ref{refinv}. 
Therefore,
$X^{\sigma_n}$ cannot be a rimpoint on both $\ell$ and $m$, as $\ell\ne m$, except possibly
if $X=X^{\sigma_n}$, but this cannot occur by Lemma \ref{perpconc}. Therefore,
$\ell$ and $m$ do not share a common perpendicular.
\end{proof}

Let $P$ and $\ell$ be a point and line that are not incident. The following lemma shows that there
exists a line incident with $P$ that is perpendicular to $\ell$.

\begin{lemma}\label{P1existence}
Let $P$ be a finite point not incident with a line $\ell$. Let $X$ and $Y$ be the two rimpoints incident with $\ell$.
Then $X^PY^P$ is a line not intersecting $\ell$ and not parallel to $\ell$, and 
$\bot(X^PY^P,\ell)$ is incident with $P$.
\end{lemma}

\begin{proof}
Let $m:=X^PY^P$. First note that $X^P$ and $Y^P$ lie on the same side of $\ell$ and hence $m$ does not intersect
$\ell$ in a finite point or rimpoint. Therefore, $m$ is not parallel to $\ell$. By Lemma \ref{diagonalpointpolar}, $XX^P\cap YY^P$
is a point and is incident with $\bot(\ell,m)$. The result follows by noting that $P=XX^P\cap YY^P$.
\end{proof}

Now if $P$ is a finite point not incident with a line $\ell$, then there cannot be two distinct lines through $P$
perpendicular to $\ell$, since otherwise $\ell$ would be the common perpendicular of two concurrent lines.
So Lemma \ref{P1existence} affirms (P1) in the definition of a metric plane in Section \ref{section:UsingBachmann}.

\begin{lemma}\label{existsperpendicular}
Let $\ell$ be a line and let $P$ be a finite point incident with $\ell$. Then there exists a unique 
line $\ell'$ through $P$ perpendicular to $\ell$.
\end{lemma}

\begin{proof}
Suppose we have a line $m$ not intersecting and not parallel to $\ell$.  (Such a line exists by (A3) and (A5)).
Let $Q$ be the point of intersection of $\ell$ and $\bot(\ell,m)$.
By Lemma \ref{midpointsexist}, there is a `midpoint' $M$ incident with $\ell$ such that $P=Q^M$. 
(If $P=Q$, then $M=P$).
Let $Y_1$ and $Y_2$ be the rimpoints of $\bot(\ell,m)$, and take the line $\ell'$ joining $Y_1^M$ and $Y_2^M$.
(If $P=Q$, then $\ell'=\bot(\ell,m)$ and we are done).
Now $Q$ is incident with $Y_1Y_2$, so by Theorem \ref{automorphism}, $P$ is incident with $\ell'$.
Therefore, by Lemma \ref{halfturnperp}, $\ell'$ is perpendicular to $\ell$. (In fact,
$\ell'=\bot(\ell,m^M)$ and so $\ell'$ is the unique line through $P$ perpendicular to $\ell$).
\end{proof} 

\begin{lemma}\label{nopoles}
If two lines have a common perpendicular, then they are not concurrent in a finite point.
\end{lemma}

\begin{proof}
Suppose $\ell$ and $m$ have a common perpendicular $n$, and suppose for a proof by contradiction that 
$\ell$ and $m$ meet in a finite point $P$. By Lemma \ref{existsperpendicular}, $n$ is not incident with $P$.
By Lemma \ref{P1existence}, $n^P$ is non-intersecting and non-parallel to $n$, and by Lemma \ref{halfturnperp},
$\ell^P$ and $m^P$ are perpendicular to $n^P$.
By Corollary \ref{halfturn_inc}, $\ell^P=\ell$ and $m^P=m$.
However, then $\ell$ and $m$ would be two distinct lines perpendicular to both $n$ and $n^P$, contradicting Lemma \ref{commonperp}. 
Therefore, $\ell$ and $m$ do not intersect in a finite point.
\end{proof}

\begin{lemma}\label{linerefauto}
Let $\ell$ be a line. Then:
\begin{enumerate}[(i)]
\item If $\cyclic$ is a cyclic order associated to $//$, then the reflection about $\ell$ reverses $\cyclic$.
\item The reflection about $\ell$  preserves the separation relation $//$;
\item If $P$ is a finite point, then $\sigma_\ell P \sigma_\ell$ is a half-turn and $P\sigma_\ell P$ is a reflection (about $\ell^P$).
\item If $\ell$ and $m$ are lines, then $\sigma_m\sigma_\ell\sigma_m$ is a reflection.
\item The reflection about $\ell$ is a perpendularity-preserving automorphism of $\mathcal{I}$, and it fixes a point $P$ if and only if
$P$ is incident with $\ell$.
\item The group $G$ generated by the half-turns is a normal subgroup of the group $H$ generated by the reflections.
\end{enumerate}
\end{lemma}

\begin{proof}\leavevmode
\begin{description}
\item[(i) and (ii)]
Let $\cyclic$ be a cyclic order associated to $//$ and let $\ell$ be a line. Consider three distinct rimpoints $X$, $Y$, $Z$ such that $\cyclic(X,Y,Z)$,
and suppose first that none of these three points are incident with $\ell$.
Without loss of generality, we may suppose that $X$ and $Z$ lie on opposite sides of $YY^{\sigma_\ell}$. 
So by Remark \ref{remark:differentsides}, we have
$\cyclic(Y,Y^{\sigma_\ell},X)$ and $\cyclic(Y^{\sigma_\ell},Y,Z)$, or $\cyclic(Y,Y^{\sigma_\ell},Z)$ and $\cyclic(Y^{\sigma_\ell},Y,X)$.
Then by definition of a reflection, $XX^{\sigma_\ell}$, $YY^{\sigma_\ell}$, $ZZ^{\sigma_\ell}$ are perpendicular to $\ell$, so no two of these lines meet in a finite point (by Lemma \ref{nopoles}). Then $X^{\sigma_\ell}$ and $Z^{\sigma_\ell}$ also lie on opposite sides of $YY^{\sigma_\ell}$; that is,
$\cyclic(Y,Y^{\sigma_\ell},X^{\sigma_\ell})$ and $\cyclic(Y^{\sigma_\ell},Y,Z^{\sigma_\ell})$, or 
$\cyclic(Y,Y^{\sigma_\ell},Z^{\sigma_\ell})$ and $\cyclic(Y^{\sigma_\ell},Y,X^{\sigma_\ell})$.
We can write the first conjunction as $\cyclic(Y^{\sigma_\ell},X^{\sigma_\ell},Y)\wedge \cyclic(Y^{\sigma_\ell},Y,Z^{\sigma_\ell})$
(by (C3)), and hence $\cyclic(Y^{\sigma_\ell}, X^{\sigma_\ell},Z^{\sigma_\ell})$ by (C4).
Therefore, $\cyclic(Z^{\sigma_\ell}, Y^{\sigma_\ell}, X^{\sigma_\ell})$ (by (C3)).

Now let us suppose at least one of our points, $X$ say, is incident with $\ell$. Then $X=X^{\sigma_\ell}$.
Let $X'$ be the other rimpoint incident with $\ell$. If $Y=X'$, then $Y=Y^{\sigma_\ell}$ and
$Z$ is not incident with $\ell$. Moreover, $Z^{\sigma_\ell}$ lies on the opposite side of $\ell$ to $Z$,
and hence $\cyclic(X,Z^{\sigma_\ell},Y)$; that is, $\cyclic(Z^{\sigma_\ell},Y^{\sigma_\ell},X^{\sigma_\ell})$
by (C3). A similar argument yields the same conclusion for when $Z=X'$.
So suppose $Y$ and $Z$ are not incident with $\ell$.

Without loss of generality, we will suppose that 
$\cyclic(X,Y,X')$. Since $Y^{\sigma_\ell}$ lies on the other side of $\ell$, it follows that $\cyclic(X',Y^{\sigma_\ell},X)$.
If $Y$ and $Z$ lie on the same side of $\ell$, then 
$X$ and $Z$ lie on opposite sides of $YY^{\sigma_\ell}$. So the argument above works in this case,
and we have $\cyclic(Z^{\sigma_\ell},Y^{\sigma_\ell},X^{\sigma_\ell})$.
So suppose $Y$ and $Z$ lie on opposite sides of $\ell$. Then $Y$ and $Z^{\sigma_\ell}$ lie
on the same side of $\ell$. Hence $X$ and $Z^{\sigma_\ell}$ lie on opposite sides of 
$YY^{\sigma_\ell}$. Since $X=X^{\sigma_\ell}$, it follows that $X$ and $Z$ lie on 
opposite sides of $YY^{\sigma_\ell}$. 
So the same argument as before yields $\cyclic(Z^{\sigma_\ell},Y^{\sigma_\ell},X^{\sigma_\ell})$.

Therefore, $\sigma_\ell$ reverses the cyclic order $\cyclic$. 
It then follows immediately that $\sigma_\ell$ preserves $//$, since 
\[
AB //CD\iff \left( \cyclic(A,B,C) \wedge \cyclic(A,D,B) \right)\vee \left( \cyclic(A,C,B)\wedge \cyclic(A,B,D) \right).
\]
for all distinct rimpoints $A$, $B$, $C$ and $D$.

\item[(iii)]
Let $P$ be a finite point and consider the map $g:=\sigma_\ell P \sigma_\ell$. 
Now $P$ is a fixed-point-free involution 
of the rimpoints such that each pair of transpositions of $P$ are separating. So
since $g$ has the same cycle-structure as $P$, and $\sigma_\ell$ 
preserves the separation relation (from (ii)), it follows from Corollary \ref{conjugationhalfturns} $g$
is a half-turn.

Now consider the map $h:=P\sigma_\ell P$. Recall that the half-turn by $P$ acts by automorphisms on $\mathcal{I}$ (by Theorem \ref{automorphism})
and so $\ell^P$ is a line whose rimpoints are the images of the rimpoints of $\ell$ under $P$.
Let $P$ be a rimpoint incident with $\ell^P$. Then $X^P$ is incident with $\ell$ and hence
$X^P=(X^P)^{\sigma_\ell}=X^h$. Therefore, $h$ fixes the rimpoints of $\ell^P$. Now take a rimpoint $Y$ not incident with $\ell^P$. 
So $Y^P$ is not incident with $\ell$ and hence the line $(Y^P)^{\sigma_\ell}(Y^P)$ is perpendicular to $\ell$ (by definition of $\sigma_\ell$).
Therefore, $(Y^P)^{\sigma_\ell P}Y$ is perpendicular to $\ell^P$ (by Lemma \ref{halfturnperp}) and so $h$ is the reflection about $\ell^P$.

\item[(iv)]
Let $L_1$ and $L_2$ be the rimpoints of $\ell$. Let $n$ be the line spanned by $L_1^{\sigma_m}$ and $L_2^{\sigma_m}$. We will
show that $\sigma_m\sigma_\ell\sigma_m=\sigma_n$. Now
\[
(L_1^{\sigma_m})^{\sigma_m\sigma_\ell\sigma_m}=L_1^{\sigma_\ell\sigma_m}=L_1^{\sigma_m}
\]
and similarly, $L_2^{\sigma_m}$ is fixed by $\sigma_m\sigma_\ell\sigma_m$. Therefore, the rimpoints of $n$ are fixed by
$\sigma_m\sigma_\ell\sigma_m$. Now suppose $X$ is a rimpoint not incident with $n$.
We will show that the line $XX^{\sigma_m\sigma_\ell\sigma_m}$ is perpendicular to $n$.
Let $Y:=X^{\sigma_m}$. Since $YY^{\sigma_\ell}\perp \ell$, there exists a line $a$ such that
$YY^{\sigma_\ell}=\bot(\ell,a)$. Let $\{A_1,A_2\}$ be the rimpoints of $a$ and let $a'$ be the line spanned by $A_1^{\sigma_m}$ and 
$A_2^{\sigma_m}$. So
$YY^{\sigma_\ell}=A_2^gL_2^h$ where $g$ and $h$ are \axial\ maps such that (i) $g$ has axis $\ell$, (ii) $h$ has axis $a$, (iii) $A_2^{g^2}=A_1$, (iv) $L_2^{h^2}=L_1$.
Without loss, we may assume that $Y=A_2^g$ and $Y^{\sigma_\ell}=L_2^h$ (since a line has two rimpoints).

Now an \axial\ map is a product of two half-turns, by Lemma \ref{axial2halfturns}. So the conjugate of an \axial\ map by a reflection
is an \axial\ map, by (iii). Hence $\sigma_mg\sigma_m$ and $\sigma_mh\sigma_m$ are \axial\ maps. Moreover, it is straight-forward to see that
$\sigma_mg\sigma_m$ has axis $n$ and $\sigma_mh\sigma_m$ has axis $a'$.
Also,
\[
(A_2^{\sigma_m})^{(\sigma_mg\sigma_m)^2}=A_2^{g^2\sigma_m}=A_1^{\sigma_m}
\]
and similarly, $(L_2^{\sigma_m})^{(\sigma_m h\sigma_m)^2}=L_2^{h^2\sigma_m}=L_1^{\sigma_m}$.
Therefore, 
\[
\bot(n,a')=(A_2^{\sigma_m})^{\sigma_mg\sigma_m}(L_2^{\sigma_m})^{\sigma_mh\sigma_m} =
A_2^{g\sigma_m}L_2^{h\sigma_m}=
Y^{\sigma_m}Y^{\sigma_\ell\sigma_m}
\]
 and hence $XX^{\sigma_m\sigma_\ell\sigma_m}=Y^{\sigma_m}Y^{\sigma_\ell\sigma_m}\perp n$.

\item[(v)]
Let $P$ be a point and $a$ be a line such that $P$ is incident with $a$.
Let $A$ and $A'$ be the rimpoints of $a$. 
By (iii), $P^{\sigma_\ell}$ is the unique point fixed by the half-turn $\sigma_\ell P\sigma_\ell$. Now 
\[
(A^{\sigma_\ell})^{\sigma_\ell P\sigma_\ell}= A^{P\sigma_\ell}=(A')^{\sigma_\ell}
\]
and similarly, $(A')^{\sigma_\ell}$ is mapped to $A^{\sigma_\ell}$ under the half-turn $\sigma_\ell P\sigma_\ell$. So
the line $A^{\sigma_\ell}(A')^{\sigma_\ell}$ is fixed by $\sigma_\ell P\sigma_\ell$ and hence it is incident with $P^{\sigma_\ell}$
(by Lemma \ref{halfturn_inc}).
Therefore, $\sigma_\ell$ preserves incidence and we write $a^{\sigma_\ell}$ for the line $A^{\sigma_\ell}(A')^{\sigma_\ell}$. 

A similar argument shows that $\sigma_\ell$ preserves perpendicularity by noting that from (iv), $a^{\sigma_\ell}$
is the axis of the reflection $\sigma_\ell\sigma_{a}\sigma_\ell$. The other lines fixed by this reflection are perpendicular to $a^{\sigma_\ell}$.
So if $b$ is perpendicular to $a$, then 
\[
(b^{\sigma_\ell})^{\sigma_\ell\sigma_{a}\sigma_\ell}=b^{\sigma_{a}\sigma_\ell}=b^{\sigma_\ell}.
\]
Hence $b^{\sigma_\ell}\perp a^{\sigma_\ell}$.
Therefore, $\sigma_\ell$ is a perpendicularity-preserving automorphism of $\mathcal{I}$. 

Now let $P$ be a point not incident with a line $\ell$, and let $L$ be a rimpoint incident with $\ell$. Then $L^P$ and $L^{P\sigma_\ell}$
lie on different sides of $\ell$. Therefore, $P^{\sigma_\ell}$ lies on the opposite of $\ell$ to $P$, because $P^{\sigma_\ell}$ is incident with $LL^{P\sigma_\ell}$.
Hence $P\ne P^\sigma_\ell$. Suppose now that $P$ is incident with $\ell$. Then $P^{\sigma_\ell}$ is also incident with $\ell$ since the rimpoints of $\ell$
are fixed by $\sigma_\ell$. By Lemma \ref{existsperpendicular}, there exists a line $m$ incident with $P$ such that $\ell\perp m$. Let $X$ be a rimpoint
on $m$. Then $X^{\sigma_\ell}=X^P$ and $P=\ell\cap XX^{\sigma_\ell}$. Since $XX^{\sigma_\ell}$ is fixed by $\sigma_\ell$, it follows that $P^{\sigma_\ell}=P$.

\item[(vi)] Follows straight from (iii).\qedhere
\end{description}
\end{proof}

\begin{lemma}\label{fixedlines}
If $m$ is a line fixed by $\sigma_\ell$, then $m=\ell$ or $m\perp \ell$.
\end{lemma}

\begin{proof}
Let $X$ and $Y$ be the rimpoints of $m$. Then $X^{\sigma_\ell}$ and $Y^{\sigma_\ell}$ are rimpoints incident with $m$
(by Lemma \ref{linerefauto}(v)).
If $X^{\sigma_\ell}=X$, then $Y^{\sigma_\ell}=Y$ and $m=\ell$, since the only rimpoints fixed by $\sigma_\ell$ are those
incident with $\ell$ (Lemma \ref{linerefauto}(v)). If $\sigma_\ell$ interchanges $X$ and $Y$, then this means precisely that $m$ is perpendicular to $\ell$.
\end{proof}

\begin{lemma}\label{incidence}
Let $P$ be a point and let $\ell$ be a line of $\mathcal{I}$. Then the following are equivalent:
\begin{enumerate}[(a)]
\item $P$ and $\ell$ are incident;
\item the half-turn about $P$ and the reflection $\sigma_\ell$ about $\ell$ commute;
\item $(P\sigma_\ell)^2=1$.
\end{enumerate}
\end{lemma}

\begin{proof}
Since half-turns and reflections are involutions, it is clear that (b) and (c) are equivalent.
By Lemma \ref{linerefauto}(v), $P$ and $\ell$ are incident if and only if $\sigma_\ell$ fixes $P$.
Since the half-turn about $P$ fixes $P$, we have that $P^{\sigma_\ell}=P$ is equivalent to $P^{P\sigma_\ell P}=P$.
Now the permutation $P\sigma_\ell P$ is the reflection about $\ell^P$, by Lemma \ref{linerefauto}(iii).
So $P$ and $\ell$ are incident if and only if $P\sigma_\ell P=\sigma_\ell$, and so (a) and (c) are equivalent.
\end{proof}

The following result allows us to more easily compute the image of a finite point under a reflection.

\begin{lemma}\label{perp_not_on_line}
Let $Q$ be a finite point not incident with a line $m$. Then $Q^{\sigma_m}=Q^P$
where $P$ is the point of intersection of $m$ and $QQ^{\sigma_m}$.
\end{lemma}

\begin{proof}
Let $M$ and $N$ be the rimpoints of $m$.
Note that $M^Q$ and $N^Q$ lie on the same side of $\ell$ and hence $M^QN^Q$ does not meet $M^{Q\sigma_m}N^{Q\sigma_m}$. 
By Lemma \ref{nopoles}, $N^QN^{Q\sigma_m}$ does not meet $M^QM^{Q\sigma_m}$ as they are both perpendicular to $m$.
So by (S4), we have $M^QN^{Q\sigma_m}// N^QM^{Q\sigma_m}$, and hence there exists a finite point $T$ such that
$T:=M^QN^{Q\sigma_m}\cap N^QM^{Q\sigma_m}$. 
Consider the hexagon
$M^QMM^{Q\sigma_m}N^QNN^{Q\sigma_m}$. The diagonal points are then $Q$, $T$, and $Q^{\sigma_m}=MM^{Q\sigma_m}\cap NN^{Q\sigma_m}$.
So $T$ lies on the line $QQ^{\sigma_m}$ by (A7). Now 
\[
T^{\sigma_m}=M^{Q\sigma_m}N^{Q}\cap N^{Q\sigma_m}M^{Q}=T
\]
and hence $T$ is incident with $m$ (by Lemma \ref{linerefauto}(v)).
Therefore, $T=QQ^{\sigma_m}\cap m =P$. So $M^{Q\sigma_m}=(M^{Q})^P$ and $N^{Q\sigma_m}=(N^{Q})^P$, and hence
\[
Q^{\sigma_m}=MM^{Q\sigma_m}\cap NN^{Q\sigma_m}=MM^{QP}\cap NN^{QP}=\left(MM^{Q}\cap NN^{Q}\right)^P=Q^P.\qedhere
\]
\end{proof}

\begin{theorem}\label{productperp}
Let $\ell$ and $m$ be two lines meeting in a finite point $P$.
Then $\ell\perp m$ if and only if $\sigma_\ell\sigma_m=P$.
%If $\ell$ and $m$ are perpendicular lines, then $\sigma_\ell\sigma_m=P$, where $P=\ell\cap m$.
\end{theorem}

\begin{proof}
Suppose firstly that $\ell$ and $m$ are perpendicular. Let $A$, $B$, $C$, $D$
be rimpoints such that $\ell=AC$ and $m=BD$. Let $Q$ be a finite point incident with $\ell$, not equal to $P$.
Since $B$ is fixed by $\sigma_m$, $Q^{\sigma_m}=Q^P$ (by Lemma \ref{perp_not_on_line}) and $\sigma_m$ is an automorphism of $\mathcal{I}$ (by Lemma \ref{linerefauto}(v)), we have
\[
B^{Q\sigma_m P}=B^{\sigma_m Q\sigma_m P}=B^{Q^{\sigma_m}P}=B^{Q^{P}P}=B^{PQ}=D^Q.
\]
Therefore, 
\[
Q^{\sigma_\ell\sigma_m P}=Q^{\sigma_m P}=(BB^Q\cap AC)^{\sigma_m P}=DD^Q\cap AC=Q.
\]
So $\sigma_\ell\sigma_mP$ fixes every point of $\ell$. Now let $X$ be a rimpoint not incident with $\ell$ and suppose without loss of generality that
$X$ lies on the opposite side of $\ell$ as $B$ does. Then $XB$ meets $\ell$ in a point $R$ and hence $X=B^R$. 
Since $B$ and $R$ are fixed by $\sigma_\ell\sigma_m P$, it follows that $X$ is also fixed by $\sigma_\ell\sigma_m P$.
So every point not incident with $\ell$ is fixed by $\sigma_\ell\sigma_m P$, and the rimpoints $A$ and $C$ are clearly
fixed by $\sigma_\ell\sigma_m P$. Therefore, $\sigma_\ell\sigma_m P=1$ as required.

Conversely, suppose $\sigma_\ell\sigma_m=P$. In particular, $\ell\ne m$ since the half-turn about $P$ acts nontrivially on the rimpoints.
Since $P$ is an involution, we also have $\sigma_m\sigma_\ell=P$. 
Now the half-turn $P$ fixes the line $m$ (by Lemma \ref{incidence}), and $\sigma_m$ fixes $m$, and so
\[
m=m^{\sigma_m\sigma_\ell}=m^{\sigma_\ell}.
\]
Therefore, by Lemma \ref{fixedlines}, $m$ is perpendicular to $\ell$.
\end{proof}

The following lemmas will be used in the next section, but they also yield fundamental results about reflections.

\begin{lemma}\label{tworefsfixed}
Let $\ell$ and $m$ be distinct lines. Then $\sigma_\ell\sigma_m$ fixes at most two rimpoints.
\end{lemma}

\begin{proof}
Suppose $\sigma_\ell\sigma_m$ fixes a rimpoint $X$. Then $X^{\sigma_\ell}=X^{\sigma_m}$ and hence
$XX^{\sigma_\ell}$ is perpendicular to $m$. By Lemma \ref{fixedlines}, either $\ell$ is perpendicular to $m$ and $X$ is incident with $\ell$,
or $\ell$ and $m$ are not concurrent in a finite point or rimpoint and $X$ is incident with the unique common perpendicular of $\ell$ and $m$. 
In both cases, there are at most two possibilities for $X$.
\end{proof}

\begin{lemma}\label{lotsfixed}
Let $\ell$, $m$, $n$ be lines concurrent in a point $P$, and let $L$ be a rimpoint on $\ell$. Let $o$ be the unique 
perpendicular line to $LL^{\sigma_m\sigma_n}$ through $P$, and let $g := \sigma_\ell \sigma_m \sigma _n \sigma_o$.
Then $g$ fixes at least four rimpoints and the finite point $P$.
\end{lemma}

\begin{proof}
Let $\ell$, $m$, $n$ be three lines concurrent in a point $P$. Let $L$ be a rimpoint on $\ell$ and let $o$ be the unique 
perpendicular line to $LL^{\sigma_m\sigma_n}$ through $P$. Now $g$ fixes $L$ because
\[
L^g=L^{\sigma_\ell\sigma_m\sigma_n\sigma_o}=(L^{\sigma_m\sigma_n})^{\sigma_o},
\]
and $g$ fixes $P$ because $P$ is fixed by each of the four reflections comprising $g$.
Therefore, $g$ fixes $L^P$ (because $g$ is an automorphism by Lemma \ref{linerefauto}(v)) and the line $\ell'$ perpendicular
to $\ell$ through $P$ (because such a line is unique). Let $\cyclic$ be a cyclic order
compatible with the separation relation $//$. Now $g$ is a product of four reflections and
so preserves $\cyclic$ (Lemma \ref{linerefauto}(i)). Let $L'$ be a rimpoint on $\ell'$ such that $\cyclic(L,L',L^P)$.
Then $\cyclic(L,(L')^g,L^P)$ and hence $L'$ is fixed by $g$ (as it lies on $\ell'$ and is on the same side
of $\ell$ as $L'$). So we have four fixed rimpoints (i.e., $L$, $L^P$, $L'$, $(L')^P$) and a fixed finite point $P$. 
\end{proof}

\section{Glide reflections}\label{section:gliderefs}

A \textbf{glide reflection} is the product $P\sigma_\ell$ where $P$ and $\ell$ are non-incident. 
It is not difficult to see that a glide reflection fixes precisely two rimpoints.
Let $P$ and $\ell$ be a point and line, and suppose $P\sigma_\ell$ fixes a rimpoint $X$. Then $X^P=X^{\sigma_\ell}$ and hence
$XX^P$ is perpendicular to $\ell$. So it follows that the only rimpoints fixed by $P\sigma_\ell$ are those lying on the unique perpendicular line to $\ell$,
through $P$.

\begin{lemma}\label{rewriteglide}
Let $A$ and $E$ be two finite points, and let $d$ be a line.
Then there exists a finite point $X$ and a line $y$ such that $AE\sigma_d=X\sigma_y$.
\end{lemma}

\begin{proof}
Let $x$ be the line joining $A$ and $E$. Suppose $A$ and $E$ are incident with $d$, that is, $x=d$. If $e$ is the perpendicular to $d$ at $E$, then
by Theorem \ref{productperp},
$AE\sigma_d=A\sigma_e$ and we are done. So we can assume that $x\ne d$: we have three cases.
\begin{description}
\item[Case 1: $x$ and $d$ meet in a finite point $Y$]
This case includes the possibility that $Y=A$ or $Y=E$.
By Theorem \ref{Buekenhout}, there is a point $X$ on $x$ such that $AEY=X$. Let $y$ be the perpendicular to $d$ incident with $Y$; so by Theorem \ref{productperp},
$Y\sigma_d=\sigma_y$. Therefore, $AE\sigma_d=XY\sigma_d=X\sigma_y$, as required.

\item[Case 2: $x$ and $d$ have a common perpendicular $q$]
Let $X=q\cap x$ and let $Q:=q\cap d$. By Theorem \ref{Buekenhout}, there is a point $Y$ on $x$ such that $AEX=Y$. By Lemma \ref{P1existence}, there exists a line $y$ perpendicular to $x$
and incident with $Y$. Then
\[
AE\sigma_d=YX\sigma_d=Y\sigma_x\sigma_q\sigma_d=\sigma_yQ.
\]
Finally, we can rewrite $\sigma_y Q$ as $Q(Q\sigma_y Q)$, which is the product of a half-turn and reflection (by Lemma \ref{linerefauto}(iii)).

\item[Case 3: $x$ and $d$ meet in a rimpoint]
That is, $x$ and $d$ are parallel. Let $a$ be the line perpendicular to $x$ and passing through $A$, and let
$e$ be another line perpendicular to $x$ and passing through $E$. Let $D$ be a finite point on $d$, and let
$x'$ be the perpendicular to $a$ passing through $D$. In particular, $x'\ne x$ and 
$x'$ is also perpendicular to $e$ (because if $M:=A'E\cap AE'$, then $\sigma_e\sigma_{x'}=\sigma_a^M\sigma_{x}^M
=M\sigma_a\sigma_{x}M=\sigma_x^M\sigma_{a}^M=\sigma_{x'}\sigma_e$).
Let $A':=a\cap x'$ and let $E':=e\cap x'$. Since $x'$ meets $d$ in a finite point,
by Case 1 above, there exists a finite point $X$ and a line $y$ such that 
$A'E'\sigma_d=X\sigma_y$.
Therefore (by Theorem \ref{productperp}),
\[
AE\sigma_d=(\sigma_a\sigma_x)(\sigma_x\sigma_e)\sigma_d=\sigma_a\sigma_e\sigma_d=(A'\sigma_{x'})(\sigma_{x'}E')\sigma_d=A'E'\sigma_d=X\sigma_y.
\]
\end{description}
In all cases, we have shown that  $AE\sigma_d=X\sigma_y$ for some point $X$ and line $y$.
\end{proof}

The following theorem can be thought of as a weakening of \emph{Pascal's Theorem on rimpoints} (A7) where we have
\textbf{at least two} internal diagonal points.

\begin{theorem}\label{A7star}
If $A$, $B$, $C$, $D$, $E$, $F$ are rimpoints such that two of the three diagonal points $P:=AB \cap DE$, $Q:=BC\cap EF$ of the hexagon exist,
then either (i) $CD$ meets $FA$ in a finite point incident with $PQ$, or (ii) $CD$, $FA$, $PQ$ have a common perpendicular.
\end{theorem}

\begin{proof}
If  $CD\cap FA$ is a finite point, then it is incident with $PQ$ by (A7). So we will suppose that
$CD$ does not meet $FA$ in a finite point. 
Let $X:=PQ\cap BF$ and $Y:=PQ\cap CE$. (Note that $X$ and $Y$ are finite points 
since $B$ and $F$ lie on different sides of $PQ$, because $B$ and $C$ lie on different sides of $PQ$, and $A$, $C$ and $F$ lie
on the same side of $PQ$; otherwise $AF$ would meet $CD$ in a finite point. Similarly, $C$ and $E$ lie
on different sides of $PQ$).
Let $U$ be the midpoint of $\overline{PY}$ and let $V$ be the midpoint of $\overline{PX}$.
So by Lemma \ref{midpointsunique}, $P^V=X$ and $P^U=Y$.
Now by Lemma \ref{automorphism}, half-turns are automorphisms, and so
\[
	(BF)^Q=B^Q F^Q=CE
\]
and hence $Y=X^Q$. By Lemma \ref{Buekenhout}, $VQU=UQV$.
Therefore,
\[
(UPVQ)^2=UP(VQU)(PVQ)=UP(UQV)(PVQ)=P^U Q P^V Q= Y QXQ=YY=1
\]
and hence $UPVQ$ is an involution.
Now $UPV$ is a half-turn (by Lemma \ref{Buekenhout}), so either $UPV=Q$ or
$(UPV)Q$ is an \axial\ element.
No \axial\ element has order $2$ (by Lemma \ref{noorder2}), and so $UPV=Q$.
Hence 
\[
E^{UV} = C^{UPV} = C^Q=B
\]
as $E^{UPU}=D^{(PU)^2}=C$.

Let $k=\bot(PQ,CD)$. 
Let $L$ and $M$ be the rimpoints of $PQ$ such that we have arranged the points in the order $L$, $P$, $Q$, $M$.
By Lemma \ref{commonperp}, there are unique \axial\ maps $g$ and $h$ such that
(i) $g$ has axis $PQ$ and $C^{g^2}=D$, (ii) $h$ has axis $CD$ and $M^{h^2}=L$,
and thus, $k=C^gM^h$. By Lemma \ref{uniqueaxial}, $g=UP$. Thus
\[
O:=C^g=C^{UP}=E^U.
\]
So $k$ is the unique perpendicular line to $PQ$ incident with the rimpoint $O$.
Similarly, the common perpendicular $k'$ of $PQ$ and $AF$ is
the unique perpendicular line to $PQ$ incident with
\[
O':=F^{VP}=B^V
\]
(n.b., the analogue of $g$ is $VP$ here). We have shown above that $B^V=E^U$ and hence
$O=O'$, and therefore, $k=k'$.
So, it follows that $PQ$, $CD$, $AF$ have a common perpendicular.
\end{proof}

\begin{theorem}\label{nocommonperp}
Let $P$ be a finite point and $\ell$ be a line. Suppose $X$, $Y$, $Z$ are rimpoints such that
$XX^{P\sigma_\ell}$, $YY^{P\sigma_\ell}$, $ZZ^{P\sigma_\ell}$ have a common perpendicular. 
%If Axiom (A7*) holds, 
Then $P$ is incident with $\ell$.
\end{theorem}

\begin{proof}
Let $o$ be the common perpendicular to $XX^{P\sigma_\ell}$, $YY^{P\sigma_\ell}$, $ZZ^{P\sigma_\ell}$.
By Lemma \ref{diagonalpointpolar}, $o$ is the Pascal line of the hexagon $XY^{P\sigma_\ell}ZX^{P\sigma_\ell}YZ^{P\sigma_\ell}$.
Since $XX^{P\sigma_\ell}$, $YY^{P\sigma_\ell}$ have a common perpendicular, $XY^{P\sigma_\ell}\cap X^{P\sigma_\ell}Y$ is a finite point $U$ (by Theorem \ref{seprelation} and Lemma \ref{nopoles}). 
Consider the hexagon $XX^PX^{P^{\sigma_\ell}}YY^PY^{P^{\sigma_\ell}}$. Since
$XX^P\cap YY^P=P$, $X^{P^{\sigma_\ell}}Y\cap XY^{P^{\sigma_\ell}}=U$ and $X^PX^{P^{\sigma_\ell}},Y^PY^{P^{\sigma_\ell}} \perp \ell$, 
it follows from Theorem \ref{A7star} that $UP\perp \ell$.

%By Lemma \ref{nocommonperp_helper}, $UP\perp \ell$. 
Similarly, $XZ^{P\sigma_\ell}\cap X^{P\sigma_\ell}Z$ is a finite point $V$ and $VP\perp \ell$. Therefore,
(by properties of perpendicularity), $UP=VP=UV=o$. So in particular, $o$ is fixed by $P\sigma_\ell$.
Let $A:=\ell\cap o$. Then $\sigma_\ell\sigma_o = A$ and hence
$P\sigma_\ell \sigma_o=PA$. Therefore, $P\sigma_\ell\sigma_o$ is either trivial (i.e., $P=A$) or axial. Now $P\sigma_\ell\sigma_o$ has three fixed points,
and so by Lemma \ref{axialfixpoints}, we must have the former case: $P\sigma_\ell \sigma_o=1$. 
Hence $\sigma_\ell \sigma_o=P$ and so $\sigma_\ell\sigma_o$ is an involution. Therefore, $\sigma_o=\sigma_{o^{\sigma_\ell}}$
and so $o$ is fixed by $\sigma_\ell$. Therefore, $\ell\perp o$ and $P$ is incident with $\ell$ (by Lemma \ref{fixedlines}).
\end{proof}

%\comment{Checked up to here}

\section{A proof of Theorem \ref{main} using a theorem of F. Bachmann}\label{section:UsingBachmann}

A \textbf{metric plane} is an incidence structure $\mathcal{M}$ of points and lines equipped with a symmetric binary relation of \emph{perpendicularity} on the lines of $\mathcal{M}$, satisfying the following axioms. In this context, a \textbf{reflection} 
in a line $a$ is a collineation of $\mathcal{M}$ that preserves the perpendicularity relation
and fixes $a$ pointwise. We call $a$ the \emph{axis} of the reflection.
\begin{enumerate}
\item[(I1)] Every pair of distinct points is incident with a unique line.
\item[(I2)] Every line is incident with at least three points.
\item[(I3)] There exist three points with the property that no line is incident with all of them.
\item[(P1)] For any point $P$ and any line $a$, there is a line $b$ incident with $P$ and 
perpendicular to $a$. If $P$ is incident with $a$, then $b$ is unique.
\item[(P2)] If $a$ is perpendicular to $b$, then there is a unique point incident with both $a$ and $b$.
\item[(M1)] Every line is the axis of at least one reflection. 
\item[(M2)] Reflections are involutions (on the points).
\item[(M3)] If three lines $a$, $b$, $c$ are concurrent, then a composition of reflections in the lines $a$, $b$ and $c$ is equal to a reflection in another line $d$.
\item[(M4)] If three lines $a$, $b$, $c$ have a common perpendicular, then a composition of reflections in the lines $a$, $b$ and $c$ is equal to a reflection in another line $d$.
\end{enumerate}

These axioms are due to Bachmann \cite[\S2.3]{Bachmann:1973aa}, except for (I3), which is due to Ewald \cite{Ewald:1971aa}
and simply ensures non-degeneracy. One consequence of these axioms is that for every line $\ell$, there is at most one 
reflection with axis $\ell$ (see \cite[p.27, Satz 3]{Bachmann:1973aa}).

The main result of our paper centres around the creation of a perpendicularity relation $\perp$ for an incidence structure
that gives rise to reflections yielding a metric plane.
Now (A1), (A2), (A3) imply (I1), (I2), (I3), and (M1), (M2) follow from the definition of a reflection.
So the important results will be the affirmations of (M3) and (M4). 

\begin{lemma}\label{M4}
An incidence structure satisfying Axioms (A1) -- (A7) plus (A10) and (A11), also satisfies (M4).
\end{lemma}

\begin{proof}
 Suppose $a$, $b$, $c$ have a common perpendicular $m$.
Let $A$, $B$ and $C$ be the intersections of $a$, $b$, $c$ with $m$, respectively.
Now 
\[
\sigma_a\sigma_b\sigma_c=\sigma_m A\sigma_m B\sigma_m C=ABC\sigma_m^3= ABC\sigma_m
\]
by Theorem \ref{productperp} and Lemma \ref{incidence}. Since $A$, $B$, $C$ are collinear, there exists a finite point $D$ on $m$ such that $ABC=D$ (by Lemma \ref{Buekenhout}).
Let $d$ be the unique line on $D$ that is perpendicular to $m$ (c.f., Lemma \ref{P1existence}). Then by Theorem \ref{productperp},
$\sigma_d \sigma_m= D$ and hence $\sigma_a\sigma_b\sigma_c= D\sigma_m=\sigma_d$.
\end{proof}

Let $\mathcal{I}$ be an incidence structure satisfying Axioms (A1) -- (A7) and (A10), (A11).
We have shown that $\mathcal{I}$ satisfies eight of the nine axioms for a metric plane.
The remaining axiom is (M3); the \emph{three reflection theorem} for a pencil of concurrent lines.

%\begin{lemma}\label{nocommonperp_helper}
%Let $P$ be a finite point and $\ell$ be a line. Suppose $X$, $Y$ are rimpoints such that
%$XY^{P\sigma_\ell}$ meets $X^{P\sigma_\ell}Y$ in a finite point $U$. 
%Then $UP$ is perpendicular to $\ell$.
%\end{lemma}
%
%\begin{proof}
%Consider the hexagon $XX^PX^{P^{\sigma_\ell}}YY^PY^{P^{\sigma_\ell}}$. Since
%$XX^P\cap YY^P=P$, $X^{P^{\sigma_\ell}}Y\cap XY^{P^{\sigma_\ell}}=U$ and $X^PX^{P^{\sigma_\ell}},Y^PY^{P^{\sigma_\ell}} \perp \ell$, 
%it follows from Theorem \ref{A7star} that $UP\perp \ell$.
%\end{proof}

\begin{lemma}\label{existsgoodline}
Let $a$, $b$, $c$ be three distinct lines concurrent in a point $P$. Then there exists a line $a'$
perpendicular to $a$ and having a common perpendicular $\ell$ with $b$ such that
$\ell$ does not intersect $c$ nor is parallel to $c$.
\end{lemma}

\begin{proof}
Let $A$, $B$, $C$ be rimpoints on $a$, $b$, $c$ respectively, such that $\cyclic(A,B,C)$.
Then $\cyclic(B,C,A^{\sigma_b})$ or $\cyclic(B,A^{\sigma_b},C)$ (by (C2));
so we can find a rimpoint $L$ such that $\cyclic(B,L,A^{\sigma_b})$ and $\cyclic(B,L,C)$ (in both cases, $L$ is between $B$ and $C$).
Therefore, $\cyclic(A,L^{\sigma_b},B)$ (as $\sigma_b$ reverses $\cyclic$, by Lemma \ref{linerefauto}(i), and fixes $B$)
and the line $\ell:=LL^{\sigma_b}$ is perpendicular to $b$ and not meeting $a$ or $c$ in a finite point
or rimpoint. So, by Lemma \ref{parallel}, $\ell$ has a common perpendicular $a'$ to $a$.
\end{proof}

\begin{lemma}\label{A7impliesM3}
An incidence structure satisfying Axioms (A1) -- (A7) plus (A10) and (A11), also satisfies (M3).
\end{lemma}

\begin{proof}
Let $\ell$, $m$, $n$ be three distinct lines concurrent in a point $P$. By Lemma \ref{existsgoodline},
there is a line $\ell'$ perpendicular to $\ell$ and having a common perpendicular $k$ with $m$ such that
$k$ does not intersect $n$ nor is parallel to $n$. That is, there is a line $n'$ perpendicular to $k$
and $n$. Let $A$ be the point of intersection of $\ell$ and $\ell'$, and let $B$ be the point of intersection
of $n$ and $n'$. Then (by Theorem \ref{productperp}),
\[
\sigma_\ell\sigma_m\sigma_n = (\sigma_\ell\sigma_{\ell'})\sigma_{\ell'}\sigma_m\sigma_{n'}(\sigma_{n'}\sigma_n)
=A(\sigma_{\ell'}\sigma_m\sigma_{n'})B.
\]
Now, $k$ is perpendicular to $\ell'$, $m$ and $n'$, so by (M4),
$\sigma_{\ell'}\sigma _m\sigma_{n'}$ is a reflection $\sigma_d$ in a line $d$ perpendicular to $k$.

Let $L$ be a rimpoint incident with $\ell$, and let $o$ be the line through $P$ perpendicular to $XX^{\sigma_m\sigma_n}$. 
Let $g=\sigma_\ell\sigma_m\sigma_n\sigma_o$. Then $g=A\sigma_dB\sigma_o=AB^{\sigma_d} \sigma_d\sigma_o$. 
Let $E:=B^{\sigma_d}$.
By Lemma \ref{rewriteglide}, we can write 
$AE\sigma_d$ as $X\sigma_y$ for some $X$ and $y$. 
Recall that $g$ fixes at least three rimpoints by Lemma \ref{lotsfixed}; call them $U$, $V$, $W$.
So $o$ is perpendicular to each of $UU^{X\sigma_y}$, $VV^{X\sigma_y}$, $WW^{X\sigma_y}$.
Hence by Theorem \ref{nocommonperp}, $X$ is incident with $y$. So $X\sigma_y=\sigma_{y'}$
where $y'$ is the line through $X$ perpendicular to $y$, and thus, $g=\sigma_{y'}\sigma_o$. However, $g$ fixes more than two rimpoints,
and so by Lemma \ref{tworefsfixed}, $g=1$ and $y'=o$. Therefore, $\sigma_\ell\sigma_m\sigma_n=\sigma_o$.
\end{proof}

\begin{theorem}\label{metricA7star}\leavevmode
The incidence structure $\mathcal{I}$, equipped with the perpendicularity relation $\perp$, is a metric plane.
\end{theorem}

\begin{proof}
The perpendicularity relation $\perp$ is symmetric from Lemma \ref{perpsym} and the reflections are automorphisms of $\mathcal{I}$
that preserve $\perp$, by Lemma \ref{linerefauto}.
Recall that (I1), (I2), (I3), (M1), (M2) hold trivially. By Lemmas \ref{M4} and \ref{A7impliesM3},
we have that (M3) and (M4) hold. The axiom (P2) follows from Theorem \ref{productperp},
and the axiom (P1) follows from Lemma \ref{P1existence}.
\end{proof}

The following theorem is a culmination of the work in Sections 15.1 and 15.2 (in particular, Theorem 9, Theorem 11 and Satz 4) of \cite{Bachmann:1973aa}; 
see also \cite[\S6]{Pambuccian:2005qv}, \cite{Pambuccian:2011bh}, \cite[\S4]{Struve:2012gf}.

\begin{theorem}[Bachmann's Theorem]\label{bachmann}
A metric plane satisfying Axiom ($\neg V^*$) and Axiom ($H$) (see below), and such that every line is an element of an end, 
is isomorphic to a Cayley-Klein hyperbolic plane over a Euclidean field with characteristic not equal to $2$.
\begin{description}
\item[Axiom ($\neg V^*$)] There are two lines $a$ and $b$ that share no point nor perpendicular in common. (Alternatively: there exists an end).
\item[Axiom ($H$)] If a point $P$ and line $\ell$ are not incident then there are at most two lines through $P$ which have neither a common point nor a common perpendicular with $\ell$.
\end{description}
\end{theorem}

An \emph{end} is a certain type of pencil of lines (see \cite[p. 335]{Struve:2012gf}), which for our purposes, is synonymous with a rimpoint (see \cite{Menger:1971ve}).

\subsection*{Proof of Theorem \ref{main}}

Let $\mathcal{I}$ be an incidence structure of points and lines satisfying (A1) through to (A7), plus axioms (A10) and (A11).
By Theorem \ref{metricA7star}, $\mathcal{I}$ is a metric plane, when equipped with the particular perpendicularity relation that we
defined in Section \ref{section:definedperp}. 
In our setting, a rimpoint is an `end' and every line of $\mathcal{I}$ is incident with two rimpoints \cite[Theorem 1.7]{Abbott:1941qv}.
So it remains to check that Axiom ($H$) holds.
Let $P$ be a point and let $\ell$ be a line that are not incident, and let $m$ be a line on $P$ that does 
not meet $\ell$ in a finite point nor has a common perpendicular with $\ell$. 
Recall that two lines are parallel if they share neither a point nor a perpendicular (Lemma \ref{parallel}),
hence, $\ell$ and $m$ are parallel.
In the introduction, we mentioned that there are at most two lines through $P$ parallel to $\ell$, by \cite[Corollary 2.4]{DeBaggis:1948rt}.
Therefore, we can directly apply Bachmann's Theorem to establish the truth of Theorem \ref{main}.

\begin{remark}
Another way to prove Theorem \ref{main} is via \emph{$H$-groups}. 
An $H$-group is an $S$-group (as defined in \cite[\S4]{Lingenberg:1979yu}) such that there
exist parallels and if $a$ and $b$ are parallel and $c$ and $d$ are parallel then
there is a element $v\in S$ with $abv$ and $cdv$ involutions. It can be shown, after a little bit of work, that our
group $H$ is an $H$-group (pardon the clash of terminology!). So we could also
use a result of Bachmann \cite[Theorem 9, p.198]{Bachmann:1973aa} to 
establish Theorem \ref{main}.
\end{remark}

\section{Abstract ovals}\label{section:abstractovals}

An \textbf{abstract oval} $\mathcal{B}=(O,S)$ is a set $O$ of elements (where
$|O|\ge 3$) together with a set of involutory permutations $S$ such that for every $A_1,
A_2, B_1, B_2 \in O$ with $\{A_1, A_2\}\cap \{B_1,B_2\}=\varnothing$ there is exactly one element $s$ of $S$
such that $A_1^s=A_2$ and $B_1^s=B_2$. Here we allow the identity permutation to be an involution.
An abstract oval gives rise to an incidence structure $\mathbb{P}(\mathcal{B})$ as follows:
\begin{description}
\item[Points]  the elements of $O$ and the involutions $S$;
\item[Lines] each line $\ell_{A,B}$ is defined by a pair of (not necessarily distinct elements) $A$, $B$ of $O$, and is incident with
$A$, $B$ and all elements of $S$ that map $A$ to $B$. (Such a line is {\bf secant} if $A \ne B$ and {\bf tangent} if $A=B$.)
\end{description}

A line $\ell$ of $\mathbb{P}(\mathcal{B})$ is said to be \textbf{Pascalian} if, for every three involutions $I,J,K$ of $S$ incident with $\ell$, the product 
$IJK$ is also an involution of $S$ incident with $\ell$. An abstract oval is said to be Pascalian if every line of its associated incidence structure is Pascalian.
A conic $\mathcal{K}$ of a Pappian projective plane $\Pi$ gives rise to a Pascalian abstract oval as follows.
Given a point $P$ of $\Pi\backslash\mathcal{K}$ and a point $X$ of $\mathcal{K}$,
we define $X^P$ to be the other point of $\mathcal{K}$ incident with $PX$ when $PX$ is a secant line,
otherwise we set $X^P:=X$ if $PX$ is tangent to $\mathcal{K}$. Each point $P$ of 
$\Pi\backslash\mathcal{K}$ then gives rise to an involution $s_P:X\mapsto X^P$,
and we let $S$ be the set of all such involutory permutations.
Then $\mathcal{B}:=(\mathcal{K},S)$ is a special abstract oval known
as an \textbf{abstract conic}. The incidence structure $\mathbb{P}(\mathcal{B})$, in this case,
is isomorphic to the incidence structure we obtain from $\Pi$ by removing the lines that do not intersect $\mathcal{K}$
(i.e., the \emph{external} lines). In general, if the incidence structure $\mathbb{P}(\mathcal{B})$ of an abstract oval
$\mathcal{B}$ is isomorphic to the incidence structure arising from some abstract conic then we say that $\mathcal{B}$
is also an abstract conic.

\begin{theorem}\label{haveabstractoval}
Let $\mathcal{I}$ be an incidence structure satisfying (A1) -- (A7), plus (A10) and (A11). Then the set of all reflections and all half-turns forms 
an abstract oval with Pascalian secant lines. 
\end{theorem}

\begin{proof}
Our underlying set is the set of rimpoints $\mathcal{O}$, and we let 
$T$ be the set of all reflections and half-turns. We will show that 
$(\mathcal{O},T)$ is an abstract oval.
Recall that in the definition of an abstract oval, we consider  $A_1,
A_2, B_1, B_2 \in \mathcal{O}$ with $\{A_1, A_2\}\cap \{B_1,B_2\}=\varnothing$, but $A_1$ and $A_2$ could be equal,
and $B_1$ and $B_2$ could be equal. So there are three cases to consider according to whether $\{A_1,A_2,B_1,B_2\}$
has size 2, 3, or 4.
Given a pair $X$, $Y$ of distinct rimpoints there is a unique reflection fixing $X$ and $Y$,
namely $\sigma_{XY}$, and no half-turn can fix a rimpoint. Next,
given a triple of distinct rimpoints $X$, $Y$, $Z$, there is a unique perpendicular $\ell$ to $YZ$ on $X$ and
so a unique reflection fixing $X$ and interchanging $Y$ and $Z$, and again, no half-turn can fix a rimpoint.
Finally, given a quadruple of distinct rimpoints $X$,$Y$,$Z$,$W$,
(i) if $XY$ and $ZW$ meet in a finite point $P$, there is a unique half-turn (namely that around $P$)
and no reflection interchanging both $X$ and $Y$ and $Z$ and $W$;
(ii) if $XY$ and $ZW$ are disjoint, there is a unique perpendicular $\ell$ to $XY$ and $ZW$, and so
a unique reflection (namely that in $\ell$) and no half-turn interchanging both $X$ and $Y$ and $Z$ and $W$.
Therefore, $(\mathcal{O},T)$ is an abstract oval.

Now suppose $X$, $Y$ are two rimpoints, and $s_1,s_2,s_3$ are involutions taking $X$ to $Y$.
If $s_1$, $s_2$, $s_3$ are reflections about the line $XY$, then $s_1s_2s_3\in T$ by Theorem \ref{M4}.
If $s_1$, $s_2$, $s_3$ are half-turns about points incident with $XY$, then $s_1s_2s_3\in T$ by Theorem \ref{Buekenhout}.
If $s_1$, $s_2$ are half-turns about points incident with $XY$, and $s_3$ is a reflection about a line $\ell$
perpendicular to $XY$, then by Theorem \ref{productperp}, there exist lines $p$ and $q$ such that $s_1=\sigma_p\sigma_{XY}$
and $s_2=\sigma_{XY}\sigma_q$. Then, $s_1s_2s_3=\sigma_p\sigma_q\sigma_\ell$, which lies in $T$ by Theorem \ref{M4}.
The final case to consider is that $s_1$ is a half-turn about a point incident with $XY$, and $s_2,s_3$ are reflections
about lines $\ell$ and $m$ perpendicular to $XY$. Again, by Theorem \ref{productperp}, there exists
a line $p$ such that $s_1=\sigma_{XY}\sigma_p$. Then, $s_1s_2s_3= \sigma_{XY}\sigma_p\sigma_\ell\sigma_m$.
By Theorem \ref{M4}, there is a line $n$ perpendicular to $XY$ such that $\sigma_p\sigma_\ell\sigma_m=\sigma_n$.
Let $N:=n\cap XY$. Then $\sigma_n\sigma_{XY}=N$ and hence $s_1s_2s_3=N\in T$.
Hence the product of three elements of $T$ taking $X$ to $Y$
 is another element of $T$ taking $X$ to $Y$. Therefore, $(\mathcal{O},T)$ has Pascalian secant lines.
\end{proof}

%Note that in the above, all we needed was for each element of $H$ to be a product of at most two involutions.
%If this could be established without recourse to (A7*), then only (A1) -- (A7), (A10) is required
%to characterise the Cayley-Klein hyperbolic planes over Euclidean fields.

%The abstract oval that we have obtained satisfies a special property that was introduced by 
%Faina {\cite{Faina:1980aa}}.
%
%\begin{definition}[Property $\mathcal{F}$]
%For every three distinct elements $A,B,C \in O$ there is an involution $\kappa\in S$ such that 
%$A^\kappa = B$, $C^\kappa = C$, $D^\kappa = D$, where $D = C^\sigma$ and $\sigma$ is the involution fixing $A$ and $B$.
%\end{definition}
%
%\begin{theorem}[G. Faina 1980, {\cite{Faina:1980aa}}]\label{Faina}
%An infinite abstract oval $(O,S)$ for which the identity permutation does not belong to $S$
%is an abstract conic if and only if it is Pascalian and has Property $\mathcal{F}$.
%\end{theorem}

 If $(O,B)$ is an abstract oval,
then we say it is \emph{regular} if $B$ is invariant under conjugation by elements of $B$.

\begin{theorem}[Bamberg, Harris, Penttila \cite{BHP}]\label{BambergHarrisPenttila}
An abstract oval $(O,S)$ such that all elements of $S$ are regular and all secant lines are
Pascalian is an abstract conic. % (i.e., a conic in a Pappian plane).
\end{theorem}

\begin{corollary}\label{abstractconic}
Let $\mathcal{I}$ be an incidence structure satisfying (A1) -- (A7), plus (A10) and (A11). Then the set of all reflections and all half-turns forms an
abstract conic. 
\end{corollary}

\begin{proof}
Notice that the set of half-turns and reflections is regular by Remark \ref{conjaxial0} (or Corollary \ref{conjugationhalfturns}), Lemma \ref{linerefauto}(iii), 
and Lemma \ref{linerefauto}(iv).
\end{proof}

To see the connection between Corollary \ref{abstractconic} and Theorem \ref{main}, we will first
realise the Cayley-Klein model of the hyperbolic plane $\mathcal{H}$ over a field $F$ with
an incidence structure arising from the elements of the group $\PGL(2,F)$ (see \cite[\S15.1]{Bachmann:1973aa}). 
To do this, we use the language of permutation group theory. The points of $\mathcal{H}$ are the points of $\PG(2,F)$ that
are internal with respect to $\mathcal{O}$. The lines of $\mathcal{H}$ are the lines of $\PG(2,F)$ that intersect $\mathcal{O}$ in
two points; the \emph{secant lines}.
By Fr\'egier's Theorem (see \cite[Theorem 25, p. 68]{Samuel:1988qf}), the harmonic homologies 
of $\PGL(2,F)$ with centre an internal point and axis the polar external line,
are precisely the fixed-point-free involutions $\hat{T}_0$ of $\PGL(2,F)$. Dually, 
the harmonic homologies of $\PGL(2,F)$ with axis a secant line and centre the external point that is the pole,
are precisely the involutions $\hat{T}_2$ of $\PGL(2,F)$ that fix two points of $\mathcal{O}$.
So we have permutational isomorphisms where the action of $\PGL(2,F)$ on $\hat{T}_0$ and $\hat{T}_2$ is by conjugation:
\begin{align*}
\left(\hat{T}_0,\PGL(2,F)\right)&\cong  
\left(\text{points of }\mathcal{H},\PGL(2,F)\right), \\
\left(\hat{T}_2,\PGL(2,F)\right)&\cong 
\left(\text{lines of }\mathcal{H},\PGL(2,F)\right). 
\end{align*}
An involution in $\hat{T}_0$ is incident with an involution in $\hat{T}_2$ if and only if they commute, and hence we obtain
an isomorphism of the incidence structures $(\hat{T}_0,\hat{T}_2)$ and $\mathcal{H}$.

So we have another proof of Theorem \ref{main}:
Theorem \ref{abstractconic} says that the half-turns and reflections of $\mathcal{O}$
are those of an abstract conic over some field $F$. 
This allows us to identify the points of $\mathcal{I}$ with fixed-point-free involutions in $\PGL(2,F)$,
 and identify the lines of $\mathcal{I}$ with 
the involutions of $\PGL(2,F)$ that fix two points.
Therefore, $\mathcal{I}$ is isomorphic to the Cayley-Klein hyperbolic plane over a field $F$,
and by \cite[Satz 4, p. 237]{Bachmann:1973aa}, $F$ is a Euclidean field.

\section{Concluding remarks}

\begin{enumerate}[1., leftmargin=*]
\item
A metric plane is said to have \emph{free mobility} if its group of motions $\Gamma$ (i.e., automorphisms that preserve perpendicularity)
acts transitively on flags; pairs $(P,\ell)$ where $P$ is a point incident with $\ell$. It turns out (see \cite[Satz 2, p.126]{Bachmann:1973aa})
that free mobility is equivalent to each of the following three properties:
\begin{enumerate}[(1)]
\item $\Gamma$ is transitive on points and lines;
\item Every two points have a midpoint, and every two lines have a bisector;
\item Every element of $\Gamma$ that has a fixed line or fixed point is a square.
\end{enumerate}
Note that the third condition above resembles the fundamental property of \axial\ maps (Lemma \ref{squares}) that enables us to define perpendicularity
of lines. The bisectors in the second condition are in analogy with midpoints for half-turns. That is, if $\ell$ and $m$ are lines, then $n$ is a bisector of
$\ell$ and $m$ if $\sigma_\ell\sigma_n=\sigma_n\sigma_m$.
We can show directly that the incidence structure $\mathcal{I}$ studied in this paper is a metric plane with free mobility:

\begin{lemma}[Bisectors exist]\label{bisectorsexist}
Every two distinct lines $\ell$ and $m$ have a (unique) bisector.
\end{lemma}

\begin{proof}
Take the rimpoints $X$, $Y$ of $\ell$ and $X'$, $Y'$ of $m$. We may suppose that we have labelled our rimpoints so that one of the following holds:
(i) $XX'$ and $YY'$ do not meet and are not parallel, (ii) $X=X'$ or (iii) $Y=Y'$.
In all cases, thee is no half-turn interchanging $X$ and $X'$, and $Y$ and $Y'$.
Since the half-turns and the reflections form the set of permutations of an abstract oval (Theorem \ref{haveabstractoval}),
there is a reflection $\sigma_n$ interchanging $X$ and $X'$ and $Y$ and $Y'$. Note that this is valid in cases (ii) and (iii).
Hence, by Lemma \ref{linerefauto}, $\ell^{\sigma_n}=m$ and hence $\sigma_n\sigma_\ell\sigma_n=\sigma_m$ (by Lemma \ref{linerefauto}(iv)).
Therefore, $\sigma_\ell\sigma_n = \sigma_n\sigma_m$ and $n$ is a bisector of $\ell$ and $m$.
\end{proof}

\item
We show that we can use Skala's Theorem for a fourth proof of Theorem \ref{main}.
In each of the three proofs that we have given, we have an intermediate step that shows that
our incidence structure $\mathcal{I}$ can be embedded into a Pappian projective plane
$\Pi$ such that collinearity of three points is preserved. Since Pappus' Theorem and Desargues' Theorem
hold true in $\Pi$, it follows that Skala's axioms (A8) and (A9) hold (if two of the diagonal points exist, then so does the third), and so by Theorem \ref{Skala},
our result follows.
%\item 
%We could also complete the proof of Theorem \ref{main} by applying a result of M\"aurer \cite{Maurer:1983cz} which
%characterises the subgroups between $\mathrm{PSL}^-(2,K)$ and $\PGL(2,K)$ (the former being the linear fractional transformations
%with determinant $\pm 1$). A group $G$ acting as permutations of $\PG(1,K)$ satisfying the following properties is a subgroup of this interval:
%\begin{enumerate}[(i)]
%\item $G$ is $2$-transitive and no nontrivial element of $G$ fixes $3$ points;
%\item $G$ has involutions that fix two points; 
%\item For every point $P$, the point stabiliser $G_P$ contains an Abelian normal subgroup $T$ acting transitively on the remaining points;
%\item Any two-point stabiliser in $G$ is Abelian.
%\end{enumerate}
%Now if we take $G$ to be the group generated by the reflections, then the first condition follows from Lemma \ref{twotransitive} and 
%Theorem \ref{s3trans}, the second is trivial and the third and fourth follow from Theorem \ref{moufang}. So there is some field
%$K$ such that $\mathrm{PSL}^-(2,K)\le G \le\PGL(2,K)$ and we proceed as we did in Section \ref{section:Tararin}.

\item Hyperbolic planes over non-Euclidean fields do not satisfy (A10), however, we do not know if there are hyperbolic planes that
have rimpoints and do not satisfy (A11). Thus we leave it as an open question whether (A11) is derivable from the other axioms.
\end{enumerate}

\section*{Acknowledgements}

The first author acknowledges the support of the Australian Research Council Future Fellowship FT120100036, and would
like to give special thanks to Sue Barwick for accommodating the first author's research visit to the University of Adelaide where
much of this work was done.

%\bibliographystyle{abbrv}
%\bibliography{references}

\end{document}